\documentclass[a4paper,12pt,reqno]{amsart}
\usepackage{graphicx}

\usepackage{amsmath,amstext,amssymb,amsopn,amsthm}
\usepackage{url}

\usepackage[margin=30mm]{geometry}
\usepackage{eucal,mathrsfs,dsfont}
\usepackage{soul}

\newtheorem{theorem}{Theorem}[section]
\newtheorem{corollary}[theorem]{Corollary}
\newtheorem{lemma}[theorem]{Lemma}
\newtheorem{proposition}[theorem]{Proposition}

\theoremstyle{definition}

\theoremstyle{remark}

\newcommand{\eps}{\varepsilon}

\newcommand{\calA}{\mathcal{A}}

\newcommand{\calJ}{\mathcal{J}}

\newcommand{\calL}{\mathcal{L}^1}

\newcommand{\R}{\mathds{R}}

\newcommand{\N}{{\mathds{N}}}

\newcommand{\RR}{\mathrm{I\kern-0.20emR}}
\newcommand{\D}{\mathrm{d}\kern0.2pt}

\DeclareMathOperator{\dist}{dist}

\title[$q$-harmonic functions of fractional Schr{\"o}dinger operator]{Gradient estimates of $q$-harmonic functions of fractional Schr{\"o}dinger operator}
\author[T. Kulczycki]{Tadeusz Kulczycki}
\thanks{The research was supported in part by NCN grant no. 2011/03/B/ST1/00423.}
\address{Tadeusz Kulczycki, Institute of Mathematics and Computer Science, Wroc{\l}aw University of Technology, Wyb. Wyspia{\'n}skiego 27, 50-370 Wroc{\l}aw, Poland.}
\email{Tadeusz.Kulczycki@pwr.wroc.pl}

\begin{document}
\begin{abstract}
We study gradient estimates of $q$-harmonic functions $u$ of the fractional Schr{\"o}dinger operator $\Delta^{\alpha/2} + q$, $\alpha \in (0,1]$ in bounded domains $D \subset \R^d$. For nonnegative $u$ we show that if $q$ is H{\"o}lder continuous of order $\eta > 1 - \alpha$ then $\nabla u(x)$ exists for any $x \in D$ and  $|\nabla u(x)| \le c u(x)/ (\dist(x,\partial D) \wedge 1)$. The exponent $1 - \alpha$ is critical i.e. when $q$ is only $1 - \alpha$ H{\"o}lder continuous $\nabla u(x)$ may not exist. The above gradient estimates are well known for $\alpha \in (1,2]$ under the assumption that $q$ belongs to the Kato class $\calJ^{\alpha - 1}$. The case $\alpha \in (0,1]$ is different. To obtain results for $\alpha \in (0,1]$ we use probabilistic methods. As a corollary, we obtain for $\alpha \in (0,1)$ that a weak solution of $\Delta^{\alpha/2}u + q u = 0$ is in fact a strong solution.
\end{abstract}

\maketitle

\section{Introduction}

Let $\alpha \in (0,2)$, $d \in \N$ and $q$ belong to the Kato class $\calJ^{\alpha}$. We say that a Borel function $u$ on $\R^d$ is {\it{$q$-harmonic}} in an open set $D \subset \R^d$ iff
\begin{equation}
\label{probab}
u(x) = E^x\left[\exp \left( \int_0^{\tau_W} q(X_s) \, ds\right) u(X_{\tau_W})\right], \quad x \in W,
\end{equation}
for every open bounded set $W$, with $\overline{W} \subset D$. Here $X_t$ is the symmetric $\alpha$-stable process in $\R^d$, $\tau_W$ the first exit time of $X_t$ from $W$, and we understand that the expectation in (\ref{probab}) is absolutely convergent.

It is possible to express the above probabilistic definition in analytic terms. Namely, it is known \cite[Theorem 5.5]{BB1999} that if $u$ is $q$-harmonic in open set $D \subset \R^d$ then $u$ is a weak solution of 
\begin{equation}
\label{weak}
\Delta^{\alpha/2}u + q u = 0, \quad \text{on} \quad D.
\end{equation}
Here $\Delta^{\alpha/2} : = -(-\Delta)^{\alpha/2}$ is the fractional Laplacian.  On the other hand if $D \subset \R^d$ is an open bounded set and $(D,q)$ is gaugeable then a weak solution of (\ref{weak}) is a $q$-harmonic function on $D$ after a modification on a set of Lebesgue measure zero (for more details see Preliminaries).

It is known \cite{BB1999} that if $u$ is $q$-harmonic in $D$ then it is continuous in $D$. The purpose of this paper is to derive further regularity results of $q$-harmonic functions. The main result is the following.

\begin{theorem}
\label{mainthm}
Let $\alpha \in (0,1]$, $d \in \N$ and $D \subset \R^d$ be an open bounded set. Assume that $q: D \to \R$ is H{\"o}lder continuous with H{\"o}lder exponent $\eta > 1 - \alpha$. Let $u$ be $q$-harmonic in $D$. If $u$ is nonnegative in $\R^d$ then $\nabla u(x)$ exists for any $x \in D$ and we have
\begin{equation}
\label{nonnegative}
|\nabla u(x)| \le c \frac{u(x)}{\delta_D(x) \wedge 1}, \quad x \in D,
\end{equation}
where $\delta_D(x) = \dist(x,\partial D)$ and $c = c(\alpha,d,\eta,q)$. 

If $u$ is not nonnegative in $\R^d$ but $\|u\|_{\infty} < \infty$ then $\nabla u(x)$ exists for any $x \in D$ and we have
\begin{equation}
\label{norm}
|\nabla u(x)| \le c \frac{\|u\|_{\infty}}{\delta_D(x) \wedge 1}, \quad x \in D,
\end{equation}
where  $c = c(\alpha,d,\eta,q)$.
\end{theorem}

The existence of $\nabla u(x)$ and similar gradient estimates are well known in the classical case for $\alpha =2$, see e.g. \cite{CZ1990} and for $\alpha \in (1,2)$, see \cite{BKN2002}. These results for $\alpha \in (1,2]$ were shown under the assumption that $q \in \calJ^{\alpha - 1}$. The biggest difference between the cases $\alpha \in (0,1]$ and $\alpha \in (1,2]$ is the fact that for $\alpha \in (0,1]$ the function $y \to |\nabla_x G_D(x,y)|$ is not integrable while for $\alpha \in (1,2]$ is integrable. Here $G_D(x,y)$ is the Green function for $\Delta^{\alpha/2}$ with Dirichlet condition on $D^c$. The fact that $y \to |\nabla_x G_D(x,y)|$ is integrable was widely used in \cite{BKN2002} for $\alpha \in (1,2)$, see e.g. \cite[Lemma 5.2]{BKN2002}. For $\alpha \in (0,1]$ more complicated method must be used. Key ingredients of the method for $\alpha \in (0,1]$ may be briefly described  as the combination of some estimates of the Green function and some self-improving estimates used in the proof of Theorem \ref{mainthm}. The proof of the estimates of the Green function is mainly probabilistic. It is based on the representation of symmetric $\alpha$-stable processes as subordinated Brownian motions and the reflection principle for the Brownian motion. This probabilistic idea is similar to the one used in the paper by B. B{\"o}ttcher, R. Schilling, J. Wang, where they study couplings of subordinated Brownian motions, see Section 2 in \cite{BSW2011}. More remarks about these probabilistic methods are at the end of Section 3.

From analytic point of view Theorem \ref{mainthm} gives some regularity results for weak solutions of (\ref{weak}). It is worth to notice that regularity results of weak solutions of equations involving the fractional Laplacian have attracted a lot of attention recently, see e.g. \cite{KNV2007}, \cite{S2011}.

One may ask whether it is possible to weaken the assumption in Theorem \ref{mainthm} that $q$ is H{\"o}lder continuous with H{\"o}lder exponent $\eta > 1 - \alpha$. It occurs that the exponent $\eta = 1 - \alpha$ is critical in the following sense.
\begin{proposition}
\label{counterexample}
For any $\alpha \in (0,1]$, $d \in \N$ and any open bounded set $D \subset \R^d$ there exists $q: D \to [0,\infty)$ which is $1- \alpha$ H{\"o}lder continuous, a function $u: \R^d \to [0,\infty)$ which is $q$-harmonic in $D$ and a point $z \in D$ such that $\nabla u(z)$ does not exist.
\end{proposition}
The proof of this proposition is based on the estimates of the Green function of the killed Brownian motion subordinated by the $\alpha/2$-stable subordinator. These estimates were obtained by R. Song in \cite{S2004}.

When a $q$-harmonic function $u$ vanishes continuously near some part of the boundary of $D$ and $D \subset \R^d$ is a bounded Lipschitz domain then the estimates obtained in Theorem \ref{mainthm} are sharp near that part of the boundary. 
\begin{theorem}
\label{sharpthm}
Let $\alpha \in (0,1]$, $d \in \N$, $D \subset \R^d$ be a bounded Lipschitz domain and $q: D \to \R$ be H{\"o}lder continuous with H{\"o}lder exponent $\eta > 1 - \alpha$. Let $V \subset \R^d$ be open and let $K$ be a compact subset of $V$. Then there exist constants $c = c(D,V,K,\alpha,q,\eta)$ and $\eps = \eps(D,V,K,\alpha,q,\eta)$ such that for every function $u: \R^d \to [0,\infty)$ which is bounded on $V$, $q$-harmonic in $D \cap V$ and vanishes in $D^c \cap V$ we have
\begin{equation*}
c^{-1} \frac{u(x)}{\delta_D(x)} \le |\nabla u(x)| \le c \frac{u(x)}{\delta_D(x)}, \quad x \in K \cap D, \quad \delta_D(x) < \eps.
\end{equation*}
\end{theorem}
Similar result was obtained for $\alpha = 2$ in \cite{BP1999} and for $\alpha \in (1,2)$ in \cite{BKN2002}, see Theorem 5.1.

As an application of our main result we obtain gradient estimates of eigenfunctions of the eigenvalue problem of the fractional Schr{\"o}dinger operator with Dirichlet boundary conditions. These estimates are formulated and proved in Section 6.

As another application of our main result we show for $\alpha \in (0,1)$ that under some assumptions on $q$ a weak solution of $\Delta^{\alpha/2} u + q u = 0$ is in fact a strong solution. Note that in the following corollary we do not have to assume that $(D,q)$ is gaugeable.
\begin{corollary}
\label{weakstrong}
Let $\alpha \in (0,1)$, $d \in \N$ and $D \subset \R^d$ be an open bounded set. Assume that $q: D \to \R$ is H{\"o}lder continuous with H{\"o}lder exponent $\eta > 1 - \alpha$ and either $u$ is nonnegative on $\R^d$ or $\|u\|_{\infty} < \infty$. If $u$ is a weak solution of (\ref{weak}) then (after a modification on a set of Lebesgue measure zero) $u$ is continuous on $D$ and it is a strong solution of (\ref{weak}).
\end{corollary}

The paper is organized as follows. Section 2 is preliminary; we collect here basic facts concerning the fractional Laplacian, the fractional Schr{\"o}dinger operator and $q$-harmonic functions. In Section 3 using probabilistic methods we obtain estimates of the Green function, which will be essential in the rest of the paper. In Section 4 the main result of the paper is proved. Section 5 contains proofs of Proposition \ref{counterexample} and Theorem \ref{sharpthm}. Section 6 concerns applications of the main result.

\section{Preliminaries}
Most of the terminology and facts presented here are taken from \cite{BB1999} and \cite{BB2000}. 
The notation $c(a,b,\ldots)$ means that $c$ is a constant depending only on $a,b,\ldots$. Constants are always positive and finite. We adopt the convention that constants may change their value from one use to another. As usual we write $x \wedge y = \min(x,y)$, $x \vee y = \max(x,y)$ for $x,y \in \R$, $\|u\|_{\infty} = \sup_{x \in \R^d} |u(x)|$ for any function $u: \R^d \to \R$, $B(x,r) = \{y \in \R^d: \, |x - y| < r\}$ for $x \in \R^d$, $r > 0$. By $e_i$, $i=1,\ldots,d$ we denote the standard basis in $\R^d$.

We denote by $(X_t,P^x)$ the standard rotation invariant ("symmetric") $\alpha$-stable process in $\R^d$, $\alpha \in (0,2]$ with the characteristic function $E^0 \exp(i \xi X_t) = \exp(- t |\xi|^{\alpha})$, $\xi \in \R^d$, $t \ge 0$. $E^x$ denotes the expectation with respect to the distribition $P^x$ of the process starting from $x \in \R^d$. We have $P^x(X_t \in A) = \int_A p(t,x,y) \, dy$, where $p(t,x,y) = p_t(y - x)$ is the transition density of $X_t$. 

For $\alpha < d$ the process $X_t$ is transient and the potential kernel of $X_t$ is given by
\begin{equation}
\label{Riesz}
K_{\alpha}(y - x) = \int_0^{\infty} p(t,x,y) \, dt = \frac{\calA(d,\alpha)}{|y - x|^{d - \alpha}}, \quad x,y \in \R^d,
\end{equation}
where $\calA(d,\gamma) = \Gamma((d - \gamma)/2)/(2^{\gamma}\pi^{d/2}|\Gamma(\gamma/2)|)$ \cite{BG1968}. When $\alpha \ge d$ the process is recurrent and it is appropriate to consider the so-called compensated kernels. Namely for $\alpha \ge d$ we put
$$
K_{\alpha}(y - x) = \int_0^{\infty} (p(t,x,y) - p(t,0,x_0)) \, dt,
$$ 
where $x_0 = 0$ for $\alpha > d = 1$, $x_0 = 1$ for $\alpha = d = 1$ and $x_0 = (0,1)$ for $\alpha = d = 2$. For $\alpha = d = 1$ we have
$$
K_{\alpha}(x) = \frac{1}{\pi} \log\left(\frac{1}{|x|}\right).
$$

For any open set $D \subset \R^d$ we put $\tau_D = \inf\{t \ge 0: \, X_t \notin D\}$ the first exit time of $X_t$ from $D$ and we denote by $p_D(t,x,y)$ the transition density of the process $X_t$ killed on exiting $D$. The transition density is given by the formula
$$
p_D(t,x,y) = p(t,x,y) - E^x(p(t - \tau_D, X(\tau_D),y), \tau_D < t), \quad \quad x,y \in D, \quad t > 0.
$$
We put $p_D(t,x,y) =0$ if $x \in D^c$ or $y \in D^c$. It is known that for each fixed $t > 0$ the function $p_D(t,\cdot,\cdot)$ is bounded and continuous on $D \times D$. 
 When $d > \alpha$ and $D \subset \R^d$ is an open set or $d = 1 \le \alpha$ and $D \subset \R^d$ is an open bounded set  we put 
$$
G_D(x,y) = \int_0^{\infty} p_D(t,x,y) \, dt, \quad x,y \in D,
$$
$G_D(x,y) = 0$ if $x \in D^c$ or $y \in D^c$. We call $G_D(x,y)$ {\it{the Green function}} for $D$. It is known that $G_D(x,\cdot)$ is continuous on $D \setminus \{x\}$. For any open bounded set $D \subset \R^d$ we define {\it{the Green operator}} $G_D$ for $D$ by 
$$
G_D f(x) = \int G_D(x,y) f(y) \, dy.
$$
We assume here that $f$ is a bounded Borel function $f: D \to \R$. We have 
$$
G_D f(x) = E^x \int_0^{\tau_D} f(X_s) \, ds.
$$

Now we briefly present basic definitions and facts concerning the fractional Laplacian and the fractional Schr{\"o}dinger operator. We follow the approach from \cite{BB1999}. We denote by $\calL$ the space of all Borel functions $f$ on $\R^d$ satisfying
$$
\int_{\R^d} \frac{|f(x)|}{(1 + |x|)^{d + \alpha}} \, dx < \infty.
$$
For $f \in \calL$ and $x \in \R^d$ we define
$$
\Delta^{\alpha/2} f(x) = \calA(d,-\alpha) \lim_{\eps \downarrow 0} 
\int_{|y - x| > \eps} \frac{f(y) - f(x)}{|y - x|^{d + \alpha}} \, dy,
$$
whenever the limit exists.

We say that a Borel function $q: \R^d \to \R$ belongs to {\it{the Kato class}} $\calJ^{\alpha}$ iff $q$ satisfies
$$
\lim_{r \downarrow 0} \sup_{x \in \R^d} \int_{|y - x| \le r} |q(y) K_{\alpha}(y - x)| \, dy = 0.
$$
For any $\alpha \in (0,2)$, $q \in \calJ^{\alpha}$ we call $\Delta^{\alpha/2} + q$ {\it{the fractional Schr{\"o}dinger}} operator.

Let $\alpha \in (0,2)$, $q \in \calJ^{\alpha}$ and $D \subset \R^d$ be an open set. For $u \in \calL$ such that $uq \in L_{\text{loc}}^1(D)$ we define the distribution $(\Delta^{\alpha/2} + q) u$ in $D$ by the formula 
$$
((\Delta^{\alpha/2} + q)u,\varphi) = (u,\Delta^{\alpha/2}\varphi + q\varphi), \quad \varphi \in C_c^{\infty}(D),
$$
(cf. Definition 3.14 in \cite{BB1999}). We will say that $u$ is {\it{a weak solution}} of 
\begin{equation}
\label{solution}
(\Delta^{\alpha/2} + q) u = 0
\end{equation}
on $D$ iff $u \in \calL$, $uq \in L_{\text{loc}}^1(D)$ and (\ref{solution}) holds in the sense of distributions in $D$. We will say that $u$ is {\it{a strong solution}} of (\ref{solution}) on $D$ iff $u \in \calL$, $uq \in L_{\text{loc}}^1(D)$ and (\ref{solution}) holds for any $x \in D$.

For $\alpha \in (0,2)$, $q \in \calJ^{\alpha}$ the multiplicative functional $e_q(t)$ is defined by $e_q(t) = \exp\left(\int_0^t q(X_s) \, ds \right)$, $t \ge 0$. For any open bounded set $D \subset \R^d$ the function
$$
u_D(x) = E^x(e_q(\tau_D))
$$
is called the {\it{gauge}} function for $(D,q)$; when it is bounded in $D$ we say that $(D,q)$ is {\it{gaugeable}}. There are several other equivalent conditions for gaugeability, in particular there is a condition in terms of the first Dirichlet eigenvalue of $\Delta^{\alpha/2} + q$ on $D$, see below.

Let $u$ be a Borel function on $\R^d$ and let $q \in \calJ^{\alpha}$. We say that $u$ is {\it{$q$-harmonic}} in an open set $D \subset \R^d$ iff
\begin{equation}
\label{qharm1}
u(x) = E^x\left[e_q(\tau_W)  u(X_{\tau_W})\right], \quad x \in W,
\end{equation}
for every bounded open set $W$ with $\overline{W} \subset D$. $u$ is called {\it{regular $q$-harmonic}} in $D$ iff
\begin{equation}
\label{qharmr}
u(x) = E^x\left[e_q(\tau_D)  u(X_{\tau_D}); \tau_D < \infty\right], \quad x \in D.
\end{equation}
We understand that the expectation in (\ref{qharm1}) and (\ref{qharmr}) is absolutely convergent. 

By the strong Markov property any regular $q$-harmonic function in $D$ is a  $q$-harmonic function in $D$. By \cite[Theorem 4.1]{BB2000} any $q$-harmonic function in $D$ is continuous in $D$. By \cite[(4.7)]{BB2000} any $q$-harmonic function in $D$ belongs to $\calL$ (when $D \ne \emptyset$). It follows that if $u$ is a $q$-harmonic function in $D$ then $uq \in L_{\text{loc}}^1(D)$.

Let $\alpha \in (0,2)$, $q \in \calJ^{\alpha}$. If $u$ is a $q$-harmonic function in an open set $D \subset \R^d$ then it is a weak solution of $(\Delta^{\alpha/2} + q) u = 0$ on $D$. Conversely assume that $D \subset \R^d$ is an open bounded set and $(D,q)$ is gaugeable. If a function $u$ is a weak solution of $(\Delta^{\alpha/2} + q) u = 0$ on $D$ then after a modification on a set of Lebesgue measure zero, $u$ is $q$-harmonic in $D$ (see \cite[Theorem 5.5]{BB1999}).

It is known that if $u$ is $q$-harmonic in open set $D \subset \R^d$ then, unless $u = 0$ on $D$ and $u = 0$ a.e. on $D^c$, $(W,q)$ is gaugeable for any open bounded set $W$ such that $\overline{W} \subset D$ (see \cite[Lemma 4.3]{BB2000}). We will often use the following representation of $q$-harmonic functions. If $u$ is $q$-harmonic in an open set $D \subset \R^d$ then for every open bounded $W$ with the exterior cone property such that $\overline{W} \subset D$ we have
\begin{equation}
\label{representation}
u(x) = E^x u(X_{\tau_W}) + G_W(qu)(x), \quad x \in D.
\end{equation}
This follows from \cite[Proposition 6.1]{BB2000} and continuity of $q$-harmonic functions.

By saying that $q: D \to \R$ is H{\"o}lder continuous with H{\"o}lder exponent $\eta > 0$ we understand that there exists a constant $c$ such that for all $x,y \in D$ we have $|q(x) - q(y)| \le c |x - y|^{\eta}$.

We finish this section with some basic information about the spectral problem for $\Delta^{\alpha/2} + q$. Assume that $D \subset \R^d$ is an open bounded set, $\alpha \in (0,2)$, $q \in \calJ^{\alpha}$. Let us consider the eigenvalue problem for the fractional Schr{\"o}dinger operator on $D$ with zero exterior condition
\begin{eqnarray}
\label{spectral1}
\Delta^{\alpha/2} \varphi + q \varphi &=& -\lambda \varphi \quad \quad \quad \text{on} \,\,\, D, \\
\label{spectral2}
 \varphi &=& 0 \quad \quad \quad \, \, \, \,  \text{on} \,\,\, D^c.
\end{eqnarray}
It is well known that for the problem (\ref{spectral1}-\ref{spectral2}) there exists a sequence of eigenvalues $\{\lambda_n\}_{n = 1}^{\infty}$ satisfying
$$
\lambda_1 < \lambda_2 \le \lambda_3 \le \ldots, \quad \quad \quad
\lim_{n \to \infty} \lambda_n = \infty,
$$
and a sequence of corresponding eigenfunctions $\{\varphi_n\}_{n = 1}^{\infty}$, which can be chosen so that they form an orthonormal basis in $L^2(D)$. All $\varphi_n$ are bounded and continuous on $D$ and $\varphi_1$ is strictly positive on $D$. It is also well known that gaugeability of $(D,q)$ is equivalent to $\lambda_1 > 0$ see \cite[Theorem 3.11]{CS1997}, cf. \cite[Theorem 4.19]{CZ1995}. We understand that (\ref{spectral2}) holds for all $x \in D^c$ and (\ref{spectral1}) holds for almost all $x \in D$. The eigenvalue problem (\ref{spectral1}-\ref{spectral2}) was studied in e.g. \cite{CS1997}, \cite{K1998} and very recently in \cite{K2011}.

For more systematic presentation of the potential theory of fractional Schr{\"o}dinger operators we refer the reader to \cite{BB1999} or to \cite{BBKRSV2009}.

\section{Estimates of the Green function}

In this section we fix $i \in \{1,\ldots,d\}$ and use the following notation
$$
H = \{(y_1,\ldots,y_d) \in \R^d: \, y_i > 0\},
$$
$$
H_0 = \{(y_1,\ldots,y_d) \in \R^d: \, y_i = 0\}.
$$
Let $R: \R^d \to \R^d$ be the reflection with respect to $H_0$. 
For any $x \in \R^d$ we put
$$
\hat{x} = R(x).
$$
We have $\hat{x} = x - 2 x_i e_i$, where $(e_1,\ldots,e_d)$ is the standard basis in $\R^d$ and $x = (x_1,\ldots,x_d)$. We say that a set $D \subset \R^d$ is {\it{symmetric with respect to $H_0$}} iff $R(D) = D$. For any set $D \subset \R^d$, which is symmetric with respect to $H_0$ we put
\begin{equation}
\label{plusminus}
D_+ = \{(y_1,\ldots,y_d) \in D: \, y_i > 0\}, \quad \quad  
D_- = \{(y_1,\ldots,y_d) \in D: \, y_i < 0\}.
\end{equation}

Let $B_t$ be the $d$-dimensional Brownian motion starting from $x \in \R^d$ (with the transition density $(4 \pi t)^{-d/2} e^{-|x - y|^2/(4t)}$) and $\eta_t$ be the $\alpha/2$-stable subordinator starting from zero, $\alpha \in (0,2)$, independent of $B_t$ ($E^{-s\eta_t} = e^{-t s^{\alpha/2}}$). It is well known that the $d$-dimensional symmetric $\alpha$-stable process $X_t$, $\alpha \in (0,2)$, starting from $x \in \R^d$ has the following representation
$$
X_t = B_{\eta_t}.
$$
Let
$$
T = \inf\{s \ge 0: \, B_s \in H_0\}.
$$
Assume that the Brownian motion $B_t$ starts from $x \in H$. We define
\begin{equation*}
    \hat{B}_t = 
    \begin{cases}
    \displaystyle
    R(B_t)  \quad \text{for} \quad t \le T \\[\medskipamount]
    \displaystyle
    B_t     \quad \quad \, \, \, \text{for} \quad t > T.
    \end{cases}
\end{equation*}
That is $\hat{B}_t$ is the mirror reflection of $B_t$ with respect to $H_0$ before $T$ and coincides with $B_t$ afterwards. It is well known that $\hat{B}_t$ is the Brownian motion starting from $\hat{x}$. Now set 
$$
\hat{X}_t = \hat{B}_{\eta_t}.
$$
$\hat{X}_t$ is the symmetric $\alpha$-stable process starting from $\hat{x}$. The above construction is taken from Section 2 in \cite{BSW2011}. When discussing probabilities of $B_{t}$, $\hat{B}_{t}$, $\eta_t$, $X_t$, $\hat{X}_t$ we will use $P_B^x$, $P_{B}^{\hat{x}}$, $P_{\eta}$, $P^x$, $P^{\hat{x}}$ respectively. 

Now we need to consider another process, which is a subordinated killed Brownian motion. We define it as follows
$$
\tilde{X}_t = (B^{H})_{\eta_t},
$$
where $B^{H}_t$ is the Brownian motion $B_t$ (starting from $x \in H$) killed on exiting $H$ and $\eta_t$ is the $\alpha/2$-stable subordinator starting from zero, independent of $B_t$. When discussing probabilities of $B_{t}^H$, $\tilde{X}_t$  we will use $P_{B^H}^{x}$, $\tilde{P}^x$ respectively.  The general theory of subordinated killed Brownian motions was studied in \cite{SV2003}. 

For any open set $D \subset \R^d$, which is symmetric with respect to $H_0$ we put
$$
\tilde{\tau}_{D_+} = \inf\{t \ge 0: \, \tilde{X}_t \notin D_+\},
$$
where $D_+$ is given by (\ref{plusminus}).

By $\tilde{p}_{D_+}(t,x,y)$ we denote the transition density of the process $\tilde{X}_t$ killed on exiting $D_+$. The idea of considering $\tilde{p}_{D_+}(t,x,y)$ comes from \cite[Section 4]{BK2004}.

Recall that $p_D(t,x,y)$ is the transition density of the symmetric $\alpha$-stable process killed on exiting  $D$. 
\begin{lemma}
\label{killed2}
Let $D \subset \R^d$ be an open set which is symmetric with respect to $H_0$. Then we have
$$
\tilde{p}_{D_+}(t,x,y) = p_D(t,x,y) - p_D(t,\hat{x},y), \quad \quad x,y \in D_+, \quad t > 0.
$$
\end{lemma}
\begin{proof}
The proof is based on the reflection principle for the Brownian motion. 
Put 
$$
\hat{\tau}_{D} = \inf\{s \ge 0: \, (\hat{B})_{\eta_s} \notin D\}.
$$
Note that 
\begin{eqnarray*}
\tau_{D} &=& \inf\{s \ge 0: \, B_{\eta_s} \notin D\},\\
\tilde{\tau}_{D_+} &=& \inf\{s \ge 0: \, (B^H)_{\eta_s} \notin D_+\},
\end{eqnarray*}
Fix $x \in D_+$, $t > 0$ and a Borel set $A \subset D_+$. We have
\begin{eqnarray}
\nonumber
&& \tilde{P}^x\left(\tilde{X}_t \in A, \tilde{\tau}_{D_+} > t\right) \\
\nonumber
&& =E_{\eta} P_{B^H}^x \left( (B^{H})_{\eta_t} \in A, \tilde{\tau}_{D_+} > t \right) \\
\nonumber
&& =E_{\eta} P_B^x \left( B_{\eta_t} \in A, \eta_t < T, \tau_{D} > t \right) \\
\label{sum1}
&& =E_{\eta} P_B^x \left( B_{\eta_t} \in A, \tau_{D} > t \right)
- E_{\eta} P_B^x \left( B_{\eta_t} \in A, \eta_t > T, \tau_{D} > t \right).
\end{eqnarray}
Note that
$$
\left\{ B_{\eta_t} \in A, \eta_t > T, \tau_{D} > t \right\}
= \left\{ \hat{B}_{\eta_t} \in A, \hat{\tau}_{D} > t \right\}.
$$
Hence (\ref{sum1}) equals
\begin{eqnarray*}
&& E_{\eta} P_B^x \left( B_{\eta_t} \in A, \tau_{D} > t \right)
- E_{\eta} P_B^{\hat{x}} \left( \hat{B}_{\eta_t} \in A, \hat{\tau}_{D} > t \right) \\
&& = P^x \left( X_t \in A, \tau_{D} > t \right) 
- P^{\hat{x}} \left( \hat{X}_t \in A, \hat{\tau}_{D} > t \right) \\
&& = \int_A p_D(t,x,y) - p_D(t,\hat{x},y) \, dy.
\end{eqnarray*}
\end{proof}

Let $D \subset \R^d$ be an open set which is symmetric with respect to $H_0$. If $d = 1 \le \alpha$ we assume additionally that $D$ is bounded. We define the Green function for $D_+$ for the process $\tilde{X}_t$ by
$$
\tilde{G}_{D_+}(x,y) = \int_{0}^{\infty} \tilde{p}_{D_+}(t,x,y) \, dt, \quad \quad \quad x,y \in D_+, 
$$
$\tilde{G}_{D_+}(x,y) = 0$ if $x \in (D_+)^c$ or $y \in (D_+)^c$. For an open bounded set $D \subset \R^d$ which is symmetric with respect to $H_0$ we define the corresponding Green operator for $D_+$ by
$$
\tilde{G}_{D_+} f(x) = \int_{D_+} \tilde{G}_{D_+}(x,y) f(y) \, dy.
$$ 
We assume here that $f$ is a bounded Borel function $f: D_+ \to \R$. Clearly we have 
$$
\tilde{G}_{D_+} f(x) = E^x \int_0^{\tilde{\tau}_{D_+}} f(\tilde{X}_s) \, ds.
$$
By Lemma \ref{killed2} we obtain the following corollary.
\begin{corollary}
\label{corollaryGreenformula}
Let $D \subset \R^d$ be an open set which is symmetric with respect to $H_0$. If $d = 1 \le \alpha$ we assume additionally that $D$ is bounded. Then we have
\begin{equation}
\label{Greenformula}
\tilde{G}_{D_+}(x,y) = G_D(x,y) - G_D(\hat{x},y), \quad \quad  \quad x,y \in D_+.
\end{equation}
\end{corollary}

\begin{lemma}
\label{coupling}
Let $B = B(0,r)$, $r > 0$. Assume that $f: B \to \R$ is Borel and bounded. Then we have
$$
G_B f(x) - G_B f(\hat{x}) = 
\int_{B_+} \tilde{G}_{B_+}(x,y) (f(y) - f(\hat{y})) \, dy.
$$
\end{lemma}
\begin{proof}
Note that $G_B(\hat{x},\hat{y}) = G_B(x,y)$ and $G_B(\hat{x},y) = G_B(x,\hat{y})$ for any $x,y \in B_+$. We have
\begin{eqnarray*}
G_B f(\hat{x}) &=&
\int_{B_+} G_B(\hat{x},y) f(y) \, dy + \int_{B_-} G_B(\hat{x},y) f(y) \, dy \\
&=& \int_{B_+} G_B(\hat{x},y) f(y) \, dy + \int_{B_+} G_B(\hat{x},\hat{y}) f(\hat{y}) \, dy \\
&=& \int_{B_+} G_B(\hat{x},y) f(y) \, dy + \int_{B_+} G_B(x,y) f(\hat{y}) \, dy.
\end{eqnarray*}
Similarly, we have
\begin{eqnarray*}
G_B f(x) &=& \int_{B_+} G_B(x,y) f(y) \, dy + 
\int_{B_+} G_B(x,\hat{y}) f(\hat{y}) \, dy \\
&=& \int_{B_+} G_B(x,y) f(y) \, dy + \int_{B_+} G_B(\hat{x},y) f(\hat{y}) \, dy.
\end{eqnarray*}
Using the above equalities and (\ref{Greenformula}) we obtain the assertion of the lemma.
\end{proof}

\begin{lemma}
\label{Greenmonotonicity}
Let $d > \alpha$ and $V \subset W \subset \R^d$ be open sets symmetric with respect to $H_0$. Then we have
$$
\tilde{G}_{V_+}(x,y) \le \tilde{G}_{W_+}(x,y), \quad \quad \quad x,y \in V_+.
$$
\end{lemma}
\begin{proof}
Let $A \subset V_+$ be a Borel bounded set. For any $x \in V_+$ we have
$$
\int_A \tilde{G}_{V_+}(x,y) \, dy 
= E^x \int_0^{\tilde{\tau}_{V_+}} 1_A(\tilde{X}_s) \, ds \le
E^x \int_0^{\tilde{\tau}_{W_+}} 1_A(\tilde{X}_s) \, ds
= \int_A \tilde{G}_{W_+}(x,y) \, dy.
$$
Now the lemma follows from (\ref{Greenformula}) and continuity of $G_D(x,\cdot)$ on $D \setminus \{x\}$ (for $D = V, W$).
\end{proof}

\begin{lemma}
\label{Greenupper1}
Let $d > \alpha \in(0,1]$. Fix $r > 0$ and put $B = B(0,r)$, $B_+ =B_+(0,r)$.  Then we have
$$
0 \le G_B(x,y) - G_B(\hat{x},y) \le c \frac{|x - \hat{x}|}{|x - y|^{d - \alpha} |\hat{x} - y|}, \quad \quad x,y \in B_+,
$$
where $c = c(d,\alpha)$.
\end{lemma}
\begin{proof}
Using Corollary \ref{corollaryGreenformula}, Lemma \ref{Greenmonotonicity} and (\ref{Riesz}) we obtain
\begin{eqnarray*}
G_B(x,y) - G_B(\hat{x},y) &=& \tilde{G}_{B_+}(x,y) \\
&\le& \tilde{G}_{\R^d_+}(x,y) \\
&=& G_{\R^d}(x,y) - G_{\R^d}(\hat{x},y) \\
&=& \calA(d,\alpha) \frac{|\hat{x} - y|^{d - \alpha} - |x - y|^{d - \alpha}}{|x - y|^{d - \alpha} |\hat{x} - y|^{d - \alpha}}.
\end{eqnarray*}

One can show that for any $p \ge q \ge 0$ and $\beta > 0$
$$
p^{\beta} - q^{\beta} \le (2 \vee 2 \beta) p^{\beta - 1} (p - q)
$$
(we omit an elementary justification of this inequality). Using this one obtains for any $x,y \in B_+$
$$
|\hat{x} - y|^{d - \alpha} - |x - y|^{d - \alpha} \le
(2 \vee (2d - 2 \alpha)) |x - \hat{x}| |\hat{x} - y|^{d - \alpha - 1},
$$ 
which implies the assertion of the lemma.
\end{proof}

Now we prove similar lower bound estimates of $G_B(x,y) - G_B(\hat{x},y)$. These lower bound estimates will be needed in the proof of Proposition \ref{counterexample}. We prove these estimates only for $x \in B_+(0,r/4)$ and $y$ belonging to some truncated cone lying inside $B_+(0,r)$. This will be enough for our purposes. The lower bound estimates are based on the results of R. Song \cite{S2004}.

\begin{lemma} 
\label{lowerboundGreen}
Let $\alpha \in (0,1]$, $d > \alpha$. Fix $r > 0$ and for any $x \in B_+(0,r/4)$ put
$$
K(r,x) = \{y = (y_1,\ldots,y_d) \in \R^d: \, r/2 > y_i > 2 |x|, |y| < \sqrt{2} y_i\}.
$$
For any $x \in B_+(0,r/4)$ and $y \in K(r,x)$ we have
$$
G_{B(0,r)}(x,y) - G_{B(0,r)}(\hat{x},y) \ge 
c \frac{|x - \hat{x}|}{|x - y|^{d - \alpha} |\hat{x} - y|},
$$
where $c = c(d,\alpha)$.
\end{lemma}
\begin{proof}
Let $B_t^D$ be the $d$-dimensional Brownian motion killed on exiting a connected bounded open set $D \subset \R^d$, $\eta_t$ - $\alpha/2$-stable subordinator starting from zero independent of $B_t^D$ and put $Z_t^D = (B^D)_{\eta_t}$. $Z_t^D$ is a Markov process with the generator $-(-\Delta |_D)^{\alpha/2}$ where $\Delta |_D$ is the Dirichlet Laplacian in $D$. The process $Z_t^D$ has been intensively studied see e.g. \cite{S2004}, \cite{SV2003}. Let $G_{D}^Z(x,y)$ be the Green function of the set $D$ of the process $Z_t^D$. Note that $\tilde{X}_t = (B^H)_{\eta_t} = Z_t^H$. It follows that 
$$
\tilde{G}_{B_+(0,r)}(x,y) \ge G_{B_+(0,r)}^Z(x,y), \quad \quad x,y \in B_+(0,r).
$$

Let us first consider the case $r = 1$. Let us fix an auxiliary set $U \subset \R^d$ such that $U$ is an open, bounded, connected set with $C^{1,1}$ boundary satisfying $B_+(0,9/10) \subset U \subset B_+(0,1)$. We need to introduce the auxiliary set $U$ because $B_+(0,1)$ is not a $C^{1,1}$ domain for $d \ge 2$ and results from \cite{S2004} which we use are for $C^{1,1}$ domains.

By \cite[Theorem 4.1]{S2004} we have for any $x, y \in U$
\begin{eqnarray}
\label{lowerGreenZ1}
G_{B(0,1)}(x,y) - G_{B(0,1)}(\hat{x},y) &=& 
\tilde{G}_{B_+(0,1)}(x,y) \\
\label{lowerGreenZ2}
&\ge& G_{B_+(0,1)}^Z(x,y) \\
\label{lowerGreenZ2a}
&\ge& G_{U}^Z(x,y) \\
\label{lowerGreenZ3}
&\ge& \left(\frac{\delta_{U}(x) \delta_{U}(y)}{|x - y|^2} \wedge 1\right) \frac{c}{|x - y|^{d - \alpha}},
\end{eqnarray}
where $c = c(d,\alpha)$. Assume that $x \in B_+(0,1/4)$ and $y \in K(1,x)$. We have
$$
1 \ge \frac{|x - \hat{x}|}{2 |\hat{x} - y|},
$$
$$
\delta_{U}(x) = \delta_{B_+(0,1)}(x) = x_i = \frac{1}{2}|x - \hat{x}|,
$$
$$
|\hat{x} - y| \ge |x - y|,
$$
$$
\delta_{U}(y) \ge \frac{y_i}{4} \ge \frac{1}{16} (|x| + |y|) \ge \frac{1}{16} |x - y|.
$$
It follows that 
$$
\frac{\delta_{U}(x) \delta_{U}(y)}{|x - y|^2} \wedge 1 \ge \frac{|x - \hat{x}|}{32 |\hat{x} - y|}.
$$
Using this and (\ref{lowerGreenZ1} - \ref{lowerGreenZ3}) we obtain for $x \in B_+(0,1/4)$, $y \in K(1,x)$
\begin{equation}
\label{lowerr1}
G_{B(0,1)}(x,y) - G_{B(0,1)}(\hat{x},y) \ge 
c \frac{|x - \hat{x}|}{|x - y|^{d - \alpha} |\hat{x} - y|},
\end{equation}
where $c = c(d,\alpha)$.

Now let $r > 0$ be arbitrary. Assume that $x \in B_+(0,r/4)$ and $y \in K(r,x)$. Note that $x/r \in B_+(0,1/4)$ and $y/r \in K(1,x/r)$. By scaling and (\ref{lowerr1}) we get 
\begin{eqnarray*}
G_{B(0,r)}(x,y) - G_{B(0,r)}(\hat{x},y)
&=& r^{\alpha - d} \left(G_{B(0,1)}\left(\frac{x}{r},\frac{y}{r}\right) - G_{B(0,1)}\left(\frac{\hat{x}}{r},\frac{y}{r}\right)\right) \\
&\ge& c r^{\alpha - d} \frac{\left|\frac{x}{r} - \frac{\hat{x}}{r}\right|}{\left|\frac{x}{r} - \frac{y}{r}\right|^{d - \alpha}  \left|\frac{\hat{x}}{r} - \frac{y}{r}\right|} \\
&=& c \frac{|x - \hat{x}|}{|x - y|^{d - \alpha} |\hat{x} - y|}.
\end{eqnarray*}
\end{proof}

To obtain estimates of $\tilde{G}_{B_+}(x,y)$ for $d = \alpha = 1$ we do not use probabilistic methods but we use the explicit formula for the Green function for an interval.

\begin{lemma}
\label{Greenupper2}
Let $d = \alpha = 1$. Fix $r > 0$. For $x \in \R$ let $\hat{x} = -x$. Put $B = (-r,r)$ and $B_+ = (0,r)$. Then for any $x,y \in B_+$ we have
\begin{equation}
\label{11}
0 \le G_B(x,y) - G_B(\hat{x},y) \le \frac{1}{\pi} \min\left(\frac{4 |x|}{|x - y|}, \log\left(\frac{2 |x + y|}{|x - y|}\right)\right).
\end{equation}
For any $x,y \in (0,r/2)$ we have
\begin{equation*}
G_B(x,y) - G_B(\hat{x},y) \ge \frac{1}{\pi} \min\left(\frac{2 |x|}{15 |x - y|}, \log\left(\frac{|x + y|}{4 |x - y|}\right)\right).
\end{equation*}
For any $x \in (0,r/4)$ and $y \in (2 x,r/2)$ we have
\begin{equation*}
G_B(x,y) - G_B(\hat{x},y) \ge \frac{2}{15 \pi} \frac{|x|}{|x - y|}.
\end{equation*}
\end{lemma}
\begin{proof}
By scaling we have
$$
G_{(-r,r)}(x,y) = G_{(-1,1)}\left(\frac{x}{r},\frac{y}{r}\right)
$$
so we may assume that $r = 1$. We have \cite{BGR1961}
$$
G_B(x,y) = \frac{1}{\pi} \log\left(\sqrt{w(x,y)} + \sqrt{1 + w(x,y)}\right),
$$
where
$$
w(x,y) = \frac{(1 - |x|^2) (1 - |y|^2)}{|x - y|^2}.
$$

Let $x,y \in B_+ = (0,1)$. Put $t_2 = w(x,y)$, $t_1 = w(\hat{x},y)$. Note that $t_2 > t_1$. It follows that 
\begin{eqnarray}
\nonumber
0 \le G_B(x,y) - G_B(\hat{x},y) 
&=& \frac{1}{\pi} \log\left(\frac{\sqrt{t_2} + \sqrt{1 + t_2}}{\sqrt{t_1} + \sqrt{1 + t_1}}\right) \\
\label{t1t2a}
&=& \frac{1}{\pi} \log\left(1 + \frac{\sqrt{t_2} + \sqrt{1 + t_2} - \sqrt{t_1} - \sqrt{1 + t_1}}{\sqrt{t_1} + \sqrt{1 + t_1}}\right).
\end{eqnarray}
It is elementary to show that $\sqrt{1 + t_2} - \sqrt{1 + t_1} \le \sqrt{t_2} - \sqrt{t_1}$. Hence (\ref{t1t2a}) is bounded from the above by
\begin{eqnarray}
\label{t1t2b}
\frac{1}{\pi}  \frac{\sqrt{t_2} + \sqrt{1 + t_2} - \sqrt{t_1} - \sqrt{1 + t_1}}{\sqrt{t_1} + \sqrt{1 + t_1}} 
&\le& \frac{2}{\pi \sqrt{t_1}} (\sqrt{t_2} - \sqrt{t_1})\\
\label{t1t2c}
&=& \frac{2 |x + y|}{\pi} \left(\frac{1}{|x - y|} - \frac{1}{|x + y|}\right)\\
\label{t1t2d}
&\le&
\frac{4|x|}{\pi |x - y|}.
\end{eqnarray}

Now assume that $x \in (0,1/4)$, $y \in (2x,1/2)$. We will show that $G_B(x,y) - G_B(\hat{x},y) \ge 2 |x|/(15 \pi |x - y|)$. Note that $4 |x|/|x - y| \le 4$. It is elementary to show that for $0 \le z \le 4$ we have $\log(1+z) \ge z/5$. Using this and (\ref{t1t2b} - \ref{t1t2d}) we obtain that (\ref{t1t2a}) is bounded from below by
\begin{equation}
\label{t1t2e}
\frac{1}{5 \pi}  \frac{\sqrt{t_2} + \sqrt{1 + t_2} - \sqrt{t_1} - \sqrt{1 + t_1}}{\sqrt{t_1} + \sqrt{1 + t_1}}.
\end{equation}
Note that $(1 - |\hat{x}|^2) (1 - |y|^2) \ge 1/2 \ge |\hat{x} - y|^2/2$, so $2t_1 \ge 1$, which implies $2\sqrt{t_1} \ge \sqrt{1 + t_1}$. Hence (\ref{t1t2e}) is bounded from below by
$$
\frac{1}{5 \pi} \frac{\sqrt{t_2} - \sqrt{t_1}}{3 \sqrt{t_1}} = 
\frac{|x + y|}{15 \pi} \left(\frac{1}{|x - y|} - \frac{1}{|x + y|}\right)
= \frac{2 |x|}{15 \pi |x - y|}.
$$

Now again let $x, y \in B_+ = (0,1)$. We have
\begin{eqnarray}
\label{t1t2f}
&&G_B(x,y) - G_B(\hat{x},y) 
= \frac{1}{\pi} \log\left(\frac{\sqrt{t_2} + \sqrt{1 + t_2}}{\sqrt{t_1} + \sqrt{1 + t_1}}\right) \\
\nonumber
&\le& \frac{1}{\pi} \log\left(\frac{2 \sqrt{1 + t_2}}{\sqrt{1 + t_1}}\right) \\
\label{t1t2h}
&=& \frac{1}{\pi} \log\left(2 \sqrt{\frac{|x+y|^2}{|x-y|^2} \left(\frac{|x-y|^2 + (1 - |x|^2)(1 - |y|^2)}{|x+y|^2 + (1 - |x|^2)(1 - |y|^2)}\right)}\right).
\end{eqnarray}
One can easily show that $|x+y|^2 + (1 - |x|^2)(1 - |y|^2) \ge 1$ and $|x-y|^2 + (1 - |x|^2)(1 - |y|^2) \le 1$. Hence (\ref{t1t2h}) is bounded from above by
$$
\frac{1}{\pi} \log\left(\frac{2 |x + y|}{|x - y|}\right).
$$ 

Now let $x, y \in (0,1/2)$. By (\ref{t1t2f}) we obtain
\begin{eqnarray}
\nonumber
&&G_B(x,y) - G_B(\hat{x},y) 
\ge  \frac{1}{\pi} \log\left(\frac{\sqrt{1 + t_2}}{2 \sqrt{1 + t_1}}\right)\\
\label{t1t2g}
&=& \frac{1}{\pi} \log\left(\frac{1}{2} \sqrt{\frac{|x+y|^2}{|x-y|^2} \left(\frac{|x-y|^2 + (1 - |x|^2)(1 - |y|^2)}{|x+y|^2 + (1 - |x|^2)(1 - |y|^2)}\right)}\right).
\end{eqnarray}
One can easily show that $|x+y|^2 + (1 - |x|^2)(1 - |y|^2) \le 2$ and $|x-y|^2 + (1 - |x|^2)(1 - |y|^2) \ge 1/2$. Hence (\ref{t1t2g}) is bounded from below by
$$
\frac{1}{\pi} \log\left(\frac{|x + y|}{4|x - y|}\right).
$$
\end{proof}

The estimates of the Green function obtained in this section are crucial in proving the main result of this paper. To get these estimates in the transient case we used probabilistic methods. There is alternative way of obtaining these estimates. Namely, one can use explicit formulas for the Green function of a ball for symmetric $\alpha$-stable processes (in fact this formula was used in the case $d = \alpha = 1$). We decided to use probabilistic methods instead of explicit formulas for two reasons. First, the probabilistic methods are much simpler. Secondly, it seems that it can be generalized to some other processes, which are subordinated Brownian motions. Especially interesting in this context is the relativistic process, which generator is $-(\sqrt{-\Delta + m^2} -m)$, see e.g. \cite{R2002}, \cite{KS2006}, \cite{CKS2012}.  This operator is called relativistic Hamiltonian and is used in some models of relativistic quantum mechanics see e.g. \cite{LS2010}. For the relativistic process the explicit formula for the Green function of a ball is not known, but it seems that the probabilistic methods from this paper could be used to study Schr{\"o}dinger equations based on the relativistic Hamiltonian $-(\sqrt{-\Delta + m^2} -m)$.

\section{Proof of the main result}
We will need the following technical lemma.
\begin{lemma}
\label{betau}
Fix $r \in (0,1]$, $i \in \{1,\ldots,d\}$ and $z = (z_1,\ldots,z_d) \in \R^d$.
For any $x = (x_1,\ldots,x_d) \in \R^d$ denote $\hat{x} = x - 2 e_i(x_i - z_i)$. Put $B = B(z,r)$. Assume that the function $f: B \to \R$ is Borel and bounded on $B$ and satisfies 
\begin{equation}
\label{halfbeta}
|f(x) - f(\hat{x})| \le A |x - z|^{\beta}, \quad \quad \quad x \in B(z,r/2),
\end{equation}
for some constants $A \ge 1$ and $\beta \ge 0$.

If $\alpha \in (0,1)$ and $\beta \in [0,1 - \alpha)$ then there exists $c = c(d,\alpha,\beta)$ such that for any $x \in B$ we have
\begin{equation}
\label{<1-alpha}
|G_Bf(x) - G_Bf(\hat{x})| \le
c A |x - z|^{\beta + \alpha} + c \frac{\sup_{y \in B} |f(y)|}{r} |x - z|^{\beta + \alpha}.
\end{equation}
If $\alpha \in (0,1]$ and $\beta > 1 - \alpha$ then there exists $c = c(d,\alpha,\beta)$ such that for any $x \in B$ we have
\begin{equation}
\label{>1-alpha}
|G_Bf(x) - G_Bf(\hat{x})| \le
c A |x - z| + c \frac{\sup_{y \in B} |f(y)|}{r} |x - z|.
\end{equation}
If $\alpha = 1$ and $\beta = 0$ then there exists $c = c(d)$ such that for any $x \in B$ we have
\begin{equation}
\label{beta0}
|G_Bf(x) - G_Bf(\hat{x})| \le
c A |x - z|^{1/2} + c \frac{\sup_{y \in B} |f(y)|}{r} |x - z|^{1/2}.
\end{equation}
\end{lemma}
\begin{proof}
Put $B_+ = \{y = (y_1,\ldots,y_d) \in B: \, y_i > 0\}$. We may assume that $z = 0$ and $x = (x_1,\ldots,x_d) \in B_+$.  By Lemma \ref{coupling} we have
$$
|G_B f(x) - G_B f(\hat{x})| \le 
\int_{B_+} \tilde{G}_{B_+}(x,y) |f(y) - f(\hat{y})| \, dy,
$$ 
where $\tilde{G}_{B_+}(x,y) = G_B(x,y) - G_B(\hat{x},y)$. We will consider two cases: case 1: $d > \alpha \in (0,1]$, case 2: $d = \alpha = 1$. We will often use the fact that $r \in (0,1]$ and $|x| < r \le 1$.

{\bf{Case 1:}} $d > \alpha \in (0,1]$. 

Note that
\begin{equation}
\label{xtildex}
|x - \hat{x}| \le 2 |x|.
\end{equation}
For any $y \in B_+$ we have 
\begin{equation}
\label{xtildey}
|\hat{x} - y| \ge |x - y|, \quad \quad \quad |\hat{x} - y| \ge |x - \hat{x}|/2.
\end{equation}
Put $U_1 = B(x,|x|) \cap \{y \in B_+: \, |y| \le r/2\}$, $U_2 = B^c(x,|x|) \cap \{y \in B_+: \, |y| \le r/2\}$, $U_3 = \{y \in B_+: \, |y| \ge r/2\}$. By Lemma \ref{Greenupper1}, (\ref{halfbeta}), (\ref{xtildex}), (\ref{xtildey}) we obtain
\begin{eqnarray*}
&& \int_{B_+} \tilde{G}_{B_+}(x,y) |f(y) - f(\hat{y})| \, dy
\le c |x - \hat{x}| \int_{B_+}  \frac{|f(y) - f(\hat{y})|}{|x - y|^{d - \alpha} |\hat{x} - y|} \, dy \\
&\le& c A \int_{U_1} \frac{|y|^{\beta} \, dy}{|x - y|^{d - \alpha}} + 
c A |x| \int_{U_2} \frac{|y|^{\beta} \, dy}{|x - y|^{d - \alpha + 1}} +
c \sup_{y \in B} |f(y)|  \int_{U_3} \frac{|x - \hat{x}| \, dy}{|x - y|^{d - \alpha} |\hat{x} - y|} \\
&=&  \text{I} + \text{II} + \text{III},
\end{eqnarray*}
where $c = c(d,\alpha)$.

For $y \in U_1$ we have $|y| \le |y - x| + |x| \le 2 |x|$. Hence
\begin{equation}
\label{estI}
\text{I} \le c A |x|^{\beta} \int_{U_1} \frac{dy}{|x - y|^{d - \alpha}} = c A |x|^{\alpha + \beta},
\end{equation}
where $c = c(d,\alpha,\beta)$. When  $|x| > r/4$ we get by (\ref{xtildey}) 
$$
\text{III} \le c \sup_{y \in B} |f(y)| \int_{B(x,2r)} \frac{dy}{|x - y|^{d - \alpha}} = c \sup_{y \in B} |f(y)| r^{\alpha} \le  c \sup_{y \in B} |f(y)| |x| r^{\alpha - 1},
$$
where $c = c(d,\alpha)$. If $|x| < r/4$ then $U_3 \subset B^c(x,r/4) \cap B(x,2r)$ and by (\ref{xtildex}), (\ref{xtildey}) we get
$$
\text{III} \le c \sup_{y \in B} |f(y)| |x| \int_{B^c(x,r/4) \cap B(x,2r)} \frac{dy}{|x - y|^{d - \alpha + 1}}  \le  c \sup_{y \in B} |f(y)| |x| r^{\alpha - 1},
$$
where $c = c(d,\alpha)$. Recall that $r \le 1$. It follows that for $x \in B_+$ we have
\begin{equation}
\label{estIII}
\text{III} \le  \frac{c}{r} \sup_{y \in B} |f(y)| |x|,
\end{equation}
where $c = c(d,\alpha)$.

For $y \in U_2$ we have $|y| \le |y - x| + |x| \le 2 |y - x|$. Hence
$$
\text{II} \le c A |x| \int_{U_2} \frac{dy}{|x - y|^{d - \alpha - \beta +1}},
$$
where $c = c(d,\alpha,\beta)$.
If $\alpha \in (0,1)$, $\beta \in [0,1 - \alpha)$ then $\text{II} \le c A |x|^{\alpha + \beta}$, where $c = c(d,\alpha,\beta)$. This, (\ref{estI}), (\ref{estIII}) imply (\ref{<1-alpha}).
If $\alpha \in (0,1]$, $\beta > 1 - \alpha$ then $\text{II} \le c A |x| r^{\alpha + \beta - 1} \le c A |x|$, where $c = c(d,\alpha,\beta)$. This, (\ref{estI}), (\ref{estIII}) imply (\ref{>1-alpha}). If $\alpha = 1$, $\beta = 0$ we have
$$
\text{II} = c A |x| \int_{U_2} \frac{dy}{|x - y|^{d}} \le
c A |x| \int_{|x|}^{2r} \rho^{-1} \, d\rho \le
c A |x| (|\log(2r)|+|\log|x||) \le
c A |x|^{1/2},
$$
where $c = c(d)$. This, (\ref{estI}), (\ref{estIII}) imply (\ref{beta0}).

{\bf{Case 2:}} $d = \alpha = 1$.

{\bf{Subcase 2a:}} $x \in (0,r/4)$. We have
\begin{eqnarray*}
&& \int_{B_+} \tilde{G}_{B_+}(x,y) |f(y) -f(\hat{y})| \, dy \\
&\le& A (2x)^{\beta} \int_0^{2x} \tilde{G}_{B_+}(x,y) \, dy +
A \int_{2x}^{r/2} y^{\beta} \tilde{G}_{B_+}(x,y) \, dy +
2(\sup_{y \in B} |f(y)|) \int_{r/2}^{r} \tilde{G}_{B_+}(x,y) \, dy \\
&=& \text{I} + \text{II} + \text{III}.
\end{eqnarray*}

By (\ref{11}) we have 
$$
\text{I} \le c A x^{\beta} \int_0^{2x} \log\left(\frac{6 x}{|x - y|}\right) \, dy \le c A x^{\beta} (x + x |\log x|),
$$
where $c = c(\beta)$. By (\ref{11}) we also have
$$
\text{II} \le c A x \int_{2x}^{r/2} \frac{y^{\beta} \, dy}{|x - y|} \le 
c A x \int_{2x}^{r/2} y^{\beta - 1} \, dy,
$$
where $c = c(\beta)$. Since $x \in (0,r/4)$ by (\ref{11}) we also get
$$
\text{III} \le c (\sup_{y \in B} |f(y)|) x \int_{r/2}^r \frac{dy}{|x - y|} \le
c (\sup_{y \in B} |f(y)|) x,
$$
where $c$ is an absolute constant.

Now, if $\beta > 0$ then
$$
\text{I} + \text{II} + \text{III} \le c A x + c \frac{x}{r} \sup_{y \in B} |f(y)|,
$$
for some $c = c(\beta)$. If $\beta = 0$ then
$$
\text{I} + \text{II} + \text{III} \le c A x^{1/2} + c \frac{x^{1/2}}{r} \sup_{y \in B} |f(y)|,
$$
where $c$ is an absolute constant.

{\bf{Subcase 2b:}} $x \in (r/4,r)$. We have 
\begin{eqnarray*}
&& \int_{B_+} \tilde{G}_{B_+}(x,y) |f(y) -f(\hat{y})| \, dy \\
&\le& c A x^{\beta} \int_0^{r/2} \tilde{G}_{B_+}(x,y) \, dy +
c(\sup_{y \in B} |f(y)|) \int_{r/2}^{r} \tilde{G}_{B_+}(x,y) \, dy \\
&=& \text{I} + \text{II},
\end{eqnarray*}
where $c = c(\beta)$.

By (\ref{11}) for any $x,y \in B_+$ we have
\begin{equation}
\label{11b}
\tilde{G}_{B_+}(x,y) \le c - \log |x - y|,
\end{equation}
where $c$ is an absolute constant.

By (\ref{11b}) we obtain
\begin{eqnarray*}
\text{I} &\le& c A x^{\beta} \left(cr - \int_0^{r/2} \log|x - y| \, dy\right) \\
&\le& c A x^{\beta} \left(cr - 2 \int_0^{r/2} \log y \, dy\right) \\
&=& c A x^{\beta} (cr + r - r \log(r/2)) \\
&\le& c A r^{\beta} (r + r |\log r|),
\end{eqnarray*}
where $c = c(\beta)$. Similarly, by (\ref{11b}) we obtain
$$
\text{II} \le c (\sup_{y \in B} |f(y)|) \left(cr - \int_{r/2}^{r} \log|x - y| \, dy\right) \le c (\sup_{y \in B} |f(y)|) (r + r |\log r|),
$$
where $c = c(\beta)$.

If $\beta > 0$ then
$$
\text{I} + \text{II} \le c A r + c \sup_{y \in B} |f(y)| \le
c A x + c \frac{x}{r} \sup_{y \in B} |f(y)|,
$$
for some $c = c(\beta)$. If $\beta = 0$ then
$$
\text{I} + \text{II} \le c A  (r + r |\log r|)+ c \sup_{y \in B} |f(y)| 
\le c A x^{1/2} + c \frac{x^{1/2}}{r} \sup_{y \in B} |f(y)|,
$$
where $c$ is an absolute constant.
\end{proof}

\begin{lemma}
\label{gradientGreen1}
Fix $r \in (0,1]$, $i \in \{1,\ldots,d\}$ and $z = (z_1,\ldots,z_d) \in \R^d$. For any $x = (x_1,\ldots,x_d) \in \R^d$ denote $\hat{x} = x - 2 e_i (x_i - z_i)$. Put $B = B(z,r)$. Assume that a Borel function $f$ satisfies
$$
|f(x) - f(y)| \le A |x - y|^{\eta}, \quad x,y \in B,
$$
for some $A > 0$ and $\eta \in (1 - \alpha,1]$. If $d > \alpha \in (0,1]$ then for any $\eps \in (0,r]$ we have
$$
\int_{B(z,\eps)} \left|\frac{\partial}{\partial z_i}G_B(z,y)\right| |f(y) - f(\hat{y})| \, dy < c A \eps^{\eta + \alpha -1},
$$
for some $c = c(d,\alpha,\eta)$. If $d = \alpha = 1$ then for any $\eps \in (0,r]$ we have
$$
\int_{B(z,\eps)} \left|\frac{\partial}{\partial z_i}G_B(z,y)\right| |f(y) - f(\hat{y})| \, dy < c A \eps^{\eta} ( 1 + |\log \eps |),
$$
for some $c = c(\eta)$. 
\end{lemma}
\begin{proof}
By \cite[Corollary 3.3]{BKN2002} we have
$$
\left|\frac{\partial}{\partial z_i}G_B(z,y)\right| \le 
d \frac{G_B(z,y)}{|z - y| \wedge r} 
= d \frac{G_B(z,y)}{|z - y|}, \quad \quad y \in B, \quad y \ne z.
$$
By the assumption on $f$ we have for $y  \in B$
$$
|f(y) - f(\hat{y})| \le A |y - \hat{y}|^{\eta} =
2^{\eta} A |y_i - z_i|^{\eta}.
$$
If $d > \alpha \in (0,1]$  for any $y \in B$ we have
$$
\left|\frac{\partial}{\partial z_i}G_B(z,y)\right| |f(y) - f(\hat{y})| \le
c A \frac{G_B(z,y)}{|z - y|^{1 - \eta}} \le
c A |z - y|^{\alpha + \eta - 1 - d},
$$
for some $c = c(d,\alpha,\eta)$.

If $d = \alpha = 1$ we obtain from \cite[Corollary 3.2]{BB2000} that for any $y \in B$ we have
$$
\left|\frac{\partial}{\partial z_i}G_B(z,y)\right| |f(y) - f(\hat{y})| \le
c A \frac{G_B(z,y)}{|z - y|^{1 - \eta}} \le
c A |z - y|^{\eta - 1} \log(1 + |z - y|^{-1}),
$$
for some $c = c(\eta)$. The above estimates imply the assertion of the lemma.
\end{proof}

\begin{lemma}
\label{gradientGreen} Let $\alpha \in (0,1]$. Fix $r \in (0,1]$, $z = (z_1,\ldots,z_d)\in \R^d$ and $i \in \{1,\ldots,d\}$. Put $B = B(z,r)$. Assume that $f$ is bounded and H{\"o}lder continuous in $B$ with H{\"o}lder exponent $\eta \in (1-\alpha,1]$, that is
$$
|f(x) - f(y)| \le A |x - y|^{\eta}, \quad \quad x,y \in B.
$$
Then $\nabla G_Bf(z)$ exists and we have
\begin{equation}
\label{GBfformula}
\frac{\partial}{\partial z_i} G_B f(z) = \int_{B_+} \frac{\partial}{\partial z_i} G_B(z,y) (f(y) - f(\hat{y})) \, dy,
\end{equation}
where $B_+ = \{(y_1,\ldots,y_d) \in B: \, y_i - z_i > 0\}$,  and $\hat{y} = y - 2(y_i - z_i) e_i$ for $y = (y_1,\ldots,y_d)$.
We also have
\begin{equation}
\label{nablaGBf}
|\nabla G_Bf(z)| \le c A r^{\eta + \alpha -1} (1 + |\log{r}|),
\end{equation}
where $c = c(d,\alpha,\eta)$.
\end{lemma}
\begin{proof}
Let $g(y) = f(y) - f(z)$. By our assumption on $f$ we obtain
\begin{equation}
\label{gestimate}
|g(y)| \le A |y - z|^\eta, \quad y \in B(z,r).
\end{equation}
Let $h \in (-r/8,r/8)$. We have
\begin{eqnarray*}
G_B f(z + e_i h) - G_B f(z) &=& 
(G_B 1_B(z + e_i h) - G_B 1_B(z)) f(z) \\
&& + G_B g(z + e_i h) - G_B g(z).
\end{eqnarray*}
By a well known \cite{G1961} explicit formula for $G_B 1_B(x)$ we get
\begin{equation*}
\lim_{h \to 0} \frac{1}{h} (G_B 1_B(z + e_i h) - G_B 1_B(z)) f(z) 
= f(z) \frac{\partial}{\partial z_i} G_B 1_B(z) = 0.
\end{equation*}
We also have
\begin{eqnarray*}
\frac{1}{h} ( G_B g(z + e_i h) - G_B g(z)) 
&=& \frac{1}{h} \int_{B(z,2 |h|)}( G_B(z + e_i h,y) - G_B(z,y)) g(y) \, dy \\
&+& \frac{1}{h} \int_{B(z,r) \setminus B(z,2|h|)}( G_B(z + e_i h,y) - G_B(z,y)) g(y) \, dy \\
&=& \text{I} + \text{II}.
\end{eqnarray*}

We will consider 2 cases: 1: $d > \alpha$, 2: $d = \alpha = 1$.

{\bf{Case 1:}} $d > \alpha$.

By (\ref{gestimate}) and the standard estimate $G_B(x,y) \le K_{\alpha}(x-y)$ we obtain
\begin{eqnarray*}
|\text{I}| &\le& A 2^{\eta} |h|^{\eta - 1} \int_{B(z,2 |h|)} G_B(z + e_i h,y) + G_B(z,y) \, dy \\
&\le& c A |h|^{\eta - 1} \int_{B(z,2 |h|)} |z + e_i h - y|^{\alpha - d} + |z - y|^{\alpha - d} \, dy \\
&\le& c A |h|^{\eta + \alpha - 1},
\end{eqnarray*}
where $c = c(d,\alpha,\eta)$.
By our assumption on $\eta$ it follows that $\lim_{h \to 0} \text{I} = 0$.

We also have
\begin{equation}
\label{IIest}
\text{II} = \int_{B(z,r) \setminus B(z,2|h|)} \frac{\partial G_B}{\partial z_i}(z + e_i h \theta,y) g(y) \, dy,
\end{equation}
where $\theta = \theta(y,z,h,i,\alpha,d,r) \in (0,1)$. Note that for $y \in B(z,r) \setminus B(z,2|h|)$ we have $|y - (z + e_i h \theta)| \ge |y - z|/2$. Using this, (\ref{gestimate}) and \cite[Corollary 3.3]{BKN2002} we obtain for $y \in B(z,r) \setminus B(z,2|h|)$ and $\theta$ as in (\ref{IIest})
$$
\left|\frac{\partial G_B}{\partial z_i}(z + e_i h \theta,y) g(y)\right| \le c A |y - z|^{\alpha + \eta - d - 1},
$$
where $c = c(\alpha,d,\eta)$. Note that by our assumption on $\eta$ the function $y \to |y - z|^{\alpha + \eta - d - 1}$ is integrable on $B = B(z,r)$. By the bounded convergence theorem we get
$$
\lim_{h \to 0} \text{II} = \int_B \frac{\partial}{\partial z_i}G_B(z,y) g(y) \, dy.
$$
It follows that 
\begin{equation}
\label{derg} 
\frac{\partial}{\partial z_i}G_B f(z) = \int_B \frac{\partial}{\partial z_i}G_B(z,y) g(y) \, dy.
\end{equation}
Note that $\frac{\partial}{\partial z_i}G_B(z,\hat{y}) = - \frac{\partial}{\partial z_i}G_B(z,y)$, $y \in B$. This and (\ref{derg}) implies (\ref{GBfformula}).

{\bf{Case 2:}} $d = \alpha = 1$.

Recall that $h \in (-r/8,r/8)$ and $r \in (0,1]$. By \cite[Corollary 3.2]{BB2000} we have
\begin{equation}
\label{logGreen}
G_B(x,y) \le c (1 + |\log|x - y||), \quad x,y \in B, x \ne y,
\end{equation}
where $c$ is an absolute constant.

By (\ref{gestimate}) we obtain
\begin{eqnarray*}
|\text{I}| &\le& A 2^{\eta} |h|^{\eta - 1} \int_{B(z,2 |h|)} G_B(z + h,y) + G_B(z,y) \, dy \\
&\le& c A |h|^{\eta - 1} \int_{B(z,2 |h|)} 1 + |\log|z + h - y|| + |\log|z - y|| \, dy \\
&\le& c A |h|^{\eta} (1 + |\log|h||),
\end{eqnarray*}
where $c = c(\eta)$. By our assumption on $\eta$ it follows that $\lim_{h \to 0} \text{I} = 0$.

We also have
\begin{equation}
\label{IIest2}
\text{II} = \int_{B(z,r) \setminus B(z,2|h|)} \frac{d G_B}{dz}(z + h \theta,y) g(y) \, dy,
\end{equation}
where $\theta = \theta(y,z,h,r) \in (0,1)$. Note that for $y \in B(z,r) \setminus B(z,2|h|)$ we have $|y - (z + h \theta)| \ge |y - z|/2$. Using this, (\ref{gestimate}), (\ref{logGreen}) and \cite[Corollary 3.3]{BKN2002} we obtain for $y \in B(z,r) \setminus B(z,2|h|)$ and $\theta$ as in (\ref{IIest2})
$$
\left|\frac{d G_B}{dz}(z + h \theta,y) g(y)\right| \le c A |y - z|^{\eta - 1} (1 + |\log|y - z||),
$$
where $c$ is an absolute constant. Note that by our assumption on $\eta$ the function $y \to |y - z|^{\eta - 1} (1 + |\log|y - z||)$ is integrable on $B = B(z,r)$. By the bounded convergence theorem we get
$$
\lim_{h \to 0} \text{II} = \int_B \frac{d}{dz}G_B(z,y) g(y) \, dy.
$$
It follows that 
\begin{equation}
\label{derg2} 
\frac{d}{dz}G_B f(z) = \int_B \frac{d}{dz}G_B(z,y) g(y) \, dy.
\end{equation}
Note that $\frac{d}{dz}G_B(z,\hat{y}) = - \frac{d}{dz}G_B(z,y)$, $y \in B$. This and (\ref{derg2}) implies (\ref{GBfformula}).

This finishes the justification of (\ref{GBfformula}) in both cases. Inequality (\ref{nablaGBf}) follows from (\ref{GBfformula}) and Lemma \ref{gradientGreen1}.
\end{proof}

\begin{lemma}
\label{alphaharmonic} Let $\alpha \in (0,2)$ and $D$ be an open set in $\R^d$. For every function $f$ which is $\alpha$-harmonic in $D$ we have
$$
|\nabla f(x)| \le d \frac{\|f\|_{\infty}}{\delta_D(x)}, \quad \quad \quad x \in D.
$$
\end{lemma}
The proof of Lemma \ref{alphaharmonic} is almost the same as the proof of Lemma 3.2 in \cite{BKN2002} and is omitted.

\begin{proof}[proof of Theorem \ref{mainthm}]
Fix arbitrary $z = (z_1,\ldots,z_d) \in D$ and $i \in \{1,\ldots,d\}$. Similarly like in Lemma \ref{betau} for any $x = (x_1,\ldots,x_d) \in \R^d$ put $\hat{x} = x - 2 e_i (x_i - z_i)$. Using \cite[Lemma 3.5]{BB2000} let us choose $r_0 = r_0(d,\alpha,\|q\|_{\infty}) \in (0,1]$ such that for any $r \in (0,r_0]$ and any ball od radius $r$ contained in $D$ the conditional gauge function for that ball is bounded from below by $1/2$ and from above by $2$.

Let $r  = (\delta_D(z) \wedge r_0)/2$ and $B = B(z,r)$. By (\ref{representation}) we get
\begin{equation}
\label{decomposition}
u(x) = f(x) + G_B(qu)(x), \quad x \in B,
\end{equation}
where $f(x) = E^x u(X_{\tau_B})$. The function $f$ is $\alpha$-harmonic on $B$.  

When $u$ is nonnegative on $\R^d$ by our choice of $r_0$ and by (2.15) in \cite{BB2000} we obtain $f(x) \le 2 u(x)$, $x \in B$. By \cite[Lemma 3.2]{BKN2002} it follows that
$$
|\nabla f(x)| \le d \frac{f(x)}{\delta_B(x)}
\le 4d \frac{u(x)}{\delta_B(z)} 
\le c \frac{u(x)}{\delta_D(z) \wedge 1}, \quad \quad x \in B(z,r/2),
$$
where $c = c(d,\alpha,\|q\|_{\infty})$.

If $u$ is not nonnegative on $\R^d$ but $\|u\|_{\infty} < \infty$ by Lemma \ref{alphaharmonic} we get
$$
|\nabla f(x)| \le d \frac{\|f\|_{\infty}}{\delta_B(x)}
\le c \frac{\|u\|_{\infty}}{\delta_D(z) \wedge 1}, \quad \quad x \in B(z,r/2),
$$
where $c = c(d,\alpha,\|q\|_{\infty})$.

Let $K = B$ when $u$ is nonnegative and $K = \R^d$ when $u$ is not nonnegative in $\R^d$ and $\|u\|_{\infty} < \infty$. It follows that for any $x \in B(z,r/2)$ we have
\begin{equation}
\label{harmonic}
|E^x u(X_{\tau_B}) - E^{\hat{x}} u(X_{\tau_B})| \le 
c \frac{\sup_{y \in K} |u(y)|}{\delta_D(z) \wedge 1} |x - \hat{x}| \le
c \frac{\sup_{y \in K} |u(y)|}{\delta_D(z) \wedge 1} |x - z|,
\end{equation}
where $c = c(d,\alpha,\|q\|_{\infty})$.

Let us consider the following inequality
\begin{equation}
\label{main}
|u(x) - u(\hat{x})| \le
c \frac{\sup_{y \in K} |u(y)|}{\delta_D(z) \wedge 1} |x - z|^{\beta}
\quad \quad x \in B(z,r/2),
\end{equation}
for some $\beta \in [0,1]$ and $c = c(d,\alpha,\beta,q,\eta)$.

Note that $|x - \hat{x}| \le 2 |x - z|$. Recall that $r \le 1/2$ so $|x - z| \le 1/2$ for $x \in B(z,r)$. If (\ref{main}) holds then for any $x \in B(z,r/2)$ we have
\begin{eqnarray}
\nonumber
|q(x) u(x) - q(\hat{x}) u(\hat{x})| &\le&
|u(x)| |q(x) - q(\hat{x})| + |q(\hat{x})| |u(x) - u(\hat{x})| \\
\nonumber
&\le& c \sup_{y \in K} |u(y)| |x - \hat{x}|^{\eta} +
c \sup_{y \in D} |q(y)| \frac{\sup_{y \in K} |u(y)|}{\delta_D(z) \wedge 1} |x - z|^{\beta} \\
\label{festimate}
&\le& c \frac{\sup_{y \in K} |u(y)|}{\delta_D(z) \wedge 1} |x - z|^{\beta \wedge \eta},
\end{eqnarray}
where $c = c(d,\alpha,\beta,q,\eta)$. Note that if $\beta < 1 - \alpha$ then $\beta \wedge \eta = \beta$ and if $\beta > 1 - \alpha$ then $\beta \wedge \eta > 1 - \alpha$. 

Assume now that (\ref{main}) holds for some ($\alpha \in (0,1)$, $\beta \in [0,1 - \alpha)$) or ($\alpha \in (0,1]$, $\beta \in (1 - \alpha,1)$) and $c = c(d,\alpha,\beta,q,\eta)$.

If $\alpha \in (0,1)$, $\beta \in [0,1 - \alpha)$ then by (\ref{festimate}) and Lemma \ref{betau} we obtain for $x \in B$
$$
|G_B(qu)(x) - G_B(qu)(\hat{x})| \le
c \frac{\sup_{y \in K} |u(y)|}{\delta_D(z) \wedge 1} |x - z|^{\beta + \alpha},
$$
where $c = c(d,\alpha,\beta,q,\eta)$.
If $\alpha \in (0,1]$, $\beta \in (1 - \alpha,1)$ then by (\ref{festimate}) and Lemma \ref{betau} we obtain for $x \in B$
$$
|G_B(qu)(x) - G_B(qu)(\hat{x})| \le
c \frac{\sup_{y \in K} |u(y)|}{\delta_D(z) \wedge 1} |x - z|,
$$
where $c = c(d,\alpha,\beta,q,\eta)$.

Joining this with (\ref{harmonic}) we obtain in view of (\ref{decomposition}) that if (\ref{main}) holds for some ($\alpha \in (0,1)$, $\beta \in [0,1 - \alpha)$) or ($\alpha \in (0,1]$, $\beta \in (1 - \alpha,1)$) and $c = c(d,\alpha,\beta,q,\eta)$ then
\begin{equation}
\label{main1}
|u(x) - u(\hat{x})| \le
c \frac{\sup_{y \in K} |u(y)|}{\delta_D(z) \wedge 1} |x - z|^{(\beta + \alpha) \wedge 1}
\quad \quad x \in B(z,r/2),
\end{equation}
for $c = c(d,\alpha,\beta,q,\eta)$.

Note that (\ref{main}) holds trivially for $\beta = 0$. Assume first that $\alpha \in (0,1)$ and $k \alpha \ne 1 - \alpha$ for any $k \in \N$. Then repeating the above procedure we obtain that (\ref{main}) holds for $\beta = 0, \alpha, 2 \alpha, \ldots$ and finally for $\beta = 1$. 

Assume now that $\alpha \in (0,1)$ and $k_0 \alpha = 1 - \alpha$ for some $k_0 \in \N$. Then we obtain that (\ref{main}) holds for $\beta = 0, \alpha, 2 \alpha, \ldots, k_0 \alpha$. Then (\ref{main}) holds for any $\beta \in [0,k_0 \alpha]$. In particular, it holds for $\beta = k_0 \alpha - \alpha/2$. By (\ref{main1}) we obtain that (\ref{main}) holds for $\beta = k_0 \alpha + \alpha/2 \in (1 - \alpha,1)$. Then, again by (\ref{main1}) we obtain that (\ref{main}) holds for $\beta = 1$. 

Finally assume that $\alpha = 1$. (\ref{beta0}) gives that (\ref{main}) holds for $\beta = 1/2$. Then by (\ref{main1}) we obtain that (\ref{main}) holds for $\beta = 1$. 

Now let us fix arbitrary $w \in D$ and put $s = (\delta_D(w) \wedge r_0)/8$. We will show that $\nabla u(w)$ exists. Let us take $x,y \in B(w,s)$. Since $z \in D$ was arbitrary one can take $z = (x + y)/2$ and choose the Cartesian coordinate system and $i$ so that $y = \hat{x} = x - 2 e_i (x_i - z_i)$. We put $r = (\delta_D(z) \wedge r_0)/2$ as before. Note that 
$$
\delta_D(z) \ge \delta_D(w) - \frac{\delta_D(w)}{8} = \frac{7 \delta_D(w)}{8} \ge 7 s, \quad \quad s \le \frac{r_0}{8}.
$$
We also have
$$
|x - z| = \frac{|x - y|}{2} \le s \le \left(\frac{\delta_D(z)}{7} \wedge \frac{r_0}{8}\right) < \frac{r}{2},
$$
so $x \in B(z,r/2)$.
On the other hand we have $\delta_D(z) \le s + \delta_D(w)$ so
$$
r = \frac{\delta_D(z) \wedge r_0}{2} \le \frac{(s + \delta_D(w)) \wedge r_0}{2} \le
\frac{s}{2} + \frac{\delta_D(w) \wedge r_0}{2} = \frac{9s}{2}, \quad \text{and} \quad |w - z| \le s.
$$
Hence $B(z,r) \subset B(w,11s/2)$ which gives $\sup_{p \in B(z,r)}|u(p)| \le \sup_{p \in B(w,11s/2)}|u(p)| $.

By (\ref{main}) for $\beta = 1$ we obtain
$$
|u(x) - u(y)| \le c \frac{\sup_{p \in K'}|u(p)|}{\delta_D(z) \wedge 1} |x - z| \le c \frac{\sup_{p \in K'}|u(p)|}{\delta_D(w) \wedge 1} |x - y|,
$$
where $c = c(d,\alpha,q,\eta)$ and $K' = B(w,11s/2)$ when $u$ is nonnegative in $\R^d$ and $K' = \R^d$ when $u$ is not nonnegative in $\R^d$ and $\|u\|_{\infty} < \infty$. Since $x, y \in B(w,s)$ were arbitrary we obtain that $q u$ is H{\"o}lder continuous with H{\"o}lder exponent $\eta \wedge 1$ in $B(w,s)$.

Using (\ref{representation}) for $W = B(w,s)$ and Lemma \ref{gradientGreen} for $B(w,s)$ we obtain that $\nabla u(w)$ exists. Since $w \in D$ was arbitrary this implies that $\nabla u$ is well defined on $D$.

Now again let us fix arbitrary $z \in D$, $i \in \{1,\ldots,d\}$ and put $r = (\delta_D(z) \wedge r_0)/2$, $B = B(z,r)$, $K = B$ when $u$ is nonnegative and $K = \R^d$ when $u$ is not nonnegative and $\|u\|_{\infty} < \infty$. 

When $u$ is nonnegative, by the Harnack principle (see \cite[Theorem 4.1]{BB2000}) we have
\begin{equation}
\label{Harnack}
\sup_{y \in K} |u(y)| = \sup_{y \in B} u(y) \le c u(z).
\end{equation}
By the proof of \cite[Theorem 4.1]{BB2000} it follows that $c = c(d,\alpha,\|q\|_{\infty})$.

Put $x = z + h e_i$, $h \in (0,r/2)$. By (\ref{main}) for $\beta = 1$ we get
$$
|u(z + h e_i) - u(z - h e_i)| \le 
c \frac{\sup_{y \in K} |u(y)|}{\delta_D(z) \wedge 1} h,
$$
where $c = c(d,\alpha,q,\eta)$. Since $i \in \{1,\ldots,d\}$ is arbitrary it follows that 
\begin{equation}
\label{firstgrad}
|\nabla u(z)| \le c \frac{\sup_{y \in K} |u(y)|}{\delta_D(z) \wedge 1}. 
\end{equation}
Of course this gives (\ref{norm}). When $u$ is nonnegative (\ref{firstgrad}) and (\ref{Harnack}) imply (\ref{nonnegative}).
\end{proof}

\section{Proof of Proposition \ref{counterexample} and Theorem \ref{sharpthm}}

First we prove Proposition \ref{counterexample}. By saying that $\nabla u(x)$ exists we understand that for each $i \in \{1,\ldots,d\}$ $\lim_{h \to 0} (u(x + h e_i) - u(x))/h$ exists and is finite. We say that a function is $0$ H{\"o}lder continuous if it is bounded and measureable.

\begin{proof}[proof of Proposition \ref{counterexample}]
Let us choose arbitrary point $w \in D$ and $r \in (0,\delta_D(w)/3)$. Put 
\begin{equation}
\label{qdef}
q(x) = 1_{B(w,r)}(x) (r^2 - |x - w|^2)^{1 - \alpha}, \quad \quad x \in \R^d.
\end{equation}
It may be easily shown that $q(x)$ is $(1 - \alpha)$ H{\"o}lder continuous. We may assume that $r$ is sufficiently small so that $(D,q)$ is gaugeable. (The fact that $(D,q)$ is gaugeable for small $r$ follows by Khasminski's lemma, see page 57 in \cite{BB1999}.) Put $u(x) = E^x(e_q(\tau_D))$, $x \in \R^d$. $u(x)$ is the gauge function for $(D,q)$. By Theorem 4.1 in \cite{BB2000} $u(x)$ is regular $q$-harmonic in $D$. Note that $u$ is continuous and bounded on $D$.

Fix $z \in \partial B(w,r)$. We may assume that the Cartesian coordinate system $(x_1,\ldots,x_d)$ is chosen so that $z = (0,\ldots,0)$ and $w = (r,0,\ldots,0)$. Let $B = B(0,r)$. We will show that $\nabla u(0)$ does not exist. On the contrary assume that $\nabla u(0)$ exists. By (\ref{representation}) we have
\begin{equation}
\label{sum}
u(x) = E^x(u(X_{\tau_B})) + G_B(qu)(x), \quad \quad x \in B.
\end{equation}
Of course $\nabla E^x(u(X_{\tau_B}))$ exists for $x \in B$ (see (10) and Lemma 3.2 in \cite{BKN2002}). Put $f_0(y) = u(0) q(y)$ and $f_1(y) = (u(y) - u(0)) q(y)$. We have 
\begin{equation}
\label{uq}
u(y) q(y) = f_0(y) + f_1(y).
\end{equation}
For any $s > 0$ put $B_+(0,s) = \{(y_1,\ldots,y_d) \in B(0,s): \, y_1 > 0\}$, $B_+ = B_+(0,r)$ and $\hat{y} = y - 2 e_1 y_1$ for $y = (y_1,\ldots,y_d) \in \R^d$. Recall that we have assumed that $\nabla u(0)$ exists. By this and boundedness of $u$ we get $|u(y) - u(0)| \le c |y|$ for some $c = c(w,z,r,D,d,\alpha,q)$ and any $y \in \R^d$. It follows that for $y \in B_+(0,r/2)$
$$
|f_1(y) - f_1(\hat{y})| = |f_1(y)| = |u(y) - u(0)| |q(y)| \le c \|q\|_{\infty} |y|,
$$
where $c = c(w,z,r,D,d,\alpha,q)$. By (\ref{>1-alpha}) for any $x \in B(0,r/2)$ we have
\begin{equation}
\label{diff1a}
|G_B f_1(x) - G_B f_1(\hat{x})| \le
c |x| + c \frac{\sup_{y \in B} |f_1(y)|}{r} |x| 
\le c |x|,
\end{equation}
for some $c = c(w,z,r,D,d,\alpha,q)$.

Now let us consider the case $\alpha \in (0,1]$, $d > \alpha$. By Lemmas \ref{coupling} and \ref{lowerboundGreen} we get for $x \in B_+(0,r/4)$
\begin{eqnarray}
\nonumber
G_B f_0(x) - G_B f_0(\hat{x}) &=&
\int_{B_+} (G_{B}(x,y) - G_{B}(\hat{x},y)) (f_0(y) - f_0(\hat{y})) \, dy \\
\nonumber
&=& u(0) \int_{B_+} (G_{B}(x,y) - G_{B}(\hat{x},y)) q(y) \, dy \\
\label{ux0}
&\ge& c u(0) \int_{K(r,x)} \frac{|x - \hat{x}|}{|x - y|^{d - \alpha} |\hat{x} -y|} q(y) \, dy,
\end{eqnarray}
where $c = c(w,z,r,D,d,\alpha,q)$ and $K(r,x)$ is defined in Lemma \ref{lowerboundGreen} for $i = 1$. Since $u$ is positive on $D$ we have $u(0) > 0$.

Note that for $x \in B_+(0,r/4)$ and $y \in K(r,x)$ we have $|x - y| \le (3/2) |y|$, $|\hat{x} - y| \le (3/2) |y|$. Hence (\ref{ux0}) is bounded from below by
$$
c u(0) |x - \hat{x}| \int_{K(r,x)} |y|^{\alpha - d - 1} q(y) \, dy,
$$ 
where $c = c(w,z,r,D,d,\alpha,q)$. One can easily show that for any $x \in B_+(0,r/4)$ and $y \in K(r,x)$ we have $q(y) \ge c |y|^{1 - \alpha}$, where $c = c(d,\alpha,r)$. Let $x = (x_1,0\ldots,0) \in B_+(0,r/4)$. It follows that 
\begin{equation}
\label{GBquotient}
\frac{G_B f_0(x) - G_B f_0(\hat{x})}{|x - \hat{x}|} \ge 
c u(0) \int_{K(r,x)} |y|^{- d} \, dy,
\end{equation}
where $c = c(w,z,r,D,d,\alpha,q)$. It is clear that if $x = (x_1,0\ldots,0)$ tends to $0$ then the right-hand side of (\ref{GBquotient}) tends to $\infty$. This, (\ref{sum}), the fact that $\nabla E^x(u(X_{\tau_B}))$ exists for $x \in B$, (\ref{uq}) and (\ref{diff1a}) give contradiction with the assumption that $\nabla u(0)$ exists.

Now we will consider the case $d = \alpha = 1$. Recall that in this case $q(y) = 1_{B(w,r)}(y) = 1_{(0,2r)}(y)$. By Lemmas \ref{coupling} and \ref{Greenupper2} we get for $x \in (0,r/4)$
\begin{eqnarray}
\nonumber
\frac{G_B f_0(x) - G_B f_0(\hat{x})}{|x - \hat{x}|} &=&
\frac{u(0)}{2 |x|} \int_{B_+} (G_{B}(x,y) - G_{B}(\hat{x},y)) q(y) \, dy \\
\label{dyy}
&\ge& \frac{u(0)}{15 \pi} \int_{2x}^{r/2} \frac{dy}{y}.
\end{eqnarray}
It is clear that if $x \to 0$ then (\ref{dyy}) tends to $\infty$. This, (\ref{sum}), the fact that $\frac{d}{dx} E^x(u(X_{\tau_B}))$ exists for $x \in B$, (\ref{uq}) and (\ref{diff1a}) give contradiction with the assumption that $u'(0)$ exists.
\end{proof}

Now we will prove lower bound gradient estimates. The idea of the proof is to some extent similar to the proof of lower bound gradient estimates in \cite{BKN2002}. The main difference is the use of Lemma \ref{uppereps} below instead of \cite[Lemma 5.4]{BKN2002}. There are essential differences in proofs of Lemma \ref{uppereps} and \cite[Lemma 5.4]{BKN2002}. The key arguments in the proof of Lemma \ref{uppereps} are based on Lemma \ref{gradientGreen}.

We will use the notation as in \cite{BKN2002}. For $x = (x_1,\ldots,x_d) \in \R^d$ we write $x = (\tilde{x},x_d)$, where $\tilde{x} = (x_1,\ldots,x_{d-1})$. In order to include the case $d = 1$ in the considerations below we make the convention that for $x \in \R$, $\tilde{x} = 0$ and we set $\R^0 = \{0\}$. 

We fix a Lipschitz function $\Gamma: \R^{d-1} \to \R$ with a Lipschitz constant $\lambda$, so that $|\Gamma(\tilde{x}) - \Gamma(\tilde{y})| \le \lambda |\tilde{x} - \tilde{y}|$ for $\tilde{x}, \tilde{y} \in \R^{d - 1}$. We put $\rho(x) = x_d - \Gamma(\tilde{x})$. $D$ denotes the special Lipschitz domain defined by $D = \{x \in \R^d: \, \rho(x) > 0\}$. The function $\rho(x)$ serves as vertical distance from $x \in D$ to $\partial{D}$. We define the ``box"
$$
\Delta(x,a,r) = \{y \in \R^d: \, 0 < \rho(y) < a, \, |\tilde{x} - \tilde{y}| < r\},
$$
where $x \in \R^d$ and $a,r > 0$. We note that $\Delta(x,a,r)$ is a Lipschitz domain. We also define the ``inverted box"
$$
\nabla(x,a,r) = \{y \in \R^d: \, -a < \rho(y) \le 0, \, |\tilde{x} - \tilde{y}| < r\}.
$$
The same symbol $\nabla$ is used for the gradient but the meaning will be clear from the context.

For $r > 0$ and $Q \in \partial D$ we set $\Delta_r = \Delta(Q,r,r)$ and $G_r = G_{\Delta_r}$. For a nonnegative function $u$ we put 
$$
u^{\Delta_r}(x) = E^x u(X_{\tau_{\Delta_r}}), \quad x \in \R^d.
$$

Fix $Q \in \partial D$ and assume that a Borel function $q$ satisfies
\begin{equation}
\label{qbox}
|q(x) - q(y)| \le A |x - y|^{\eta},
\end{equation}
for some $A > 0$, $\eta \in (1-\alpha,1]$ and all $x,y \in \Delta_{s_0}$ for some $s_0 \in (0,1]$.

Now we will repeat the assertion of Lemma 5.3 \cite{BKN2002}. Note that the assertion of Lemma 5.3 in \cite{BKN2002} holds for all $\alpha \in (0,2)$ under the condition that $q 1_{\Delta_{s_0}} \in \calJ^{\alpha}$ for some $s_0 \in (0,1]$. This condition follows from (\ref{qbox}). 

For every $\eps > 0$ there exists a constant $r_0 = r_0(d,\lambda,\alpha,\eta, q,s_0,\eps) \le s_0 \le 1$ such that if $r \in (0,r_0]$ and $u: \R^d \to [0,\infty)$ is $q$-harmonic and bounded in $\Delta_r$ then
\begin{equation}
\label{eps1}
(1 - \eps) u^{\Delta_r}(x) \le u(x) \le (1 + \eps) u^{\Delta_r}(x), \quad \quad x \in \R^d,
\end{equation}
and
\begin{equation}
\label{eps2}
G_r(|q| u)(x) \le \eps u^{\Delta_r}(x), \quad \quad x \in \R^d,
\end{equation}

\begin{lemma}
\label{uppereps}
Let $\alpha \in (0,1]$ and $\eps \in (0,1/2]$. There exist constants $c = c(d,\alpha,\eta,q)$ and $\kappa = \kappa(d,\lambda,\alpha,\eta,q,r_0,s_0,\eps) \le r_0$ such that if $0 < r \le \kappa$, $u: \R^d \to [0,\infty)$ is $q$-harmonic and bounded in $\Delta_r$ then
$$
|\nabla G_r(qu)(x)| \le \eps c \frac{u(x)}{\delta_{\Delta_r}(x)}, \quad \quad x \in \Delta_r.
$$
\end{lemma}
\begin{proof}
Let us choose 
$$
\kappa = \max\{s \in (0,r_0]: \, \sup_{0 < a \le s} a^{\eta + \alpha - 1} (1 + |\log a|) \le \eps\}.
$$
Fix $r \in (0,\kappa]$ and $x_0 \in \Delta_r$. Note that $\delta_{\Delta_r}(x_0) \le r \le 1$. Let $B = B(x_0,\delta_{\Delta_r}(x_0)/2)$. We have
\begin{equation}
\label{Grsplit}
G_r(qu)(x_0) = G_B(qu)(x_0) + E^{x_0} G_r(qu)(X_{\tau_B}).
\end{equation}
We will estimate gradient of two terms on the right-hand side of (\ref{Grsplit}) separately.

Let $P_B(x,z)$, $x \in B$, $z \in \text{int}(B^c)$ be the Poisson kernel for $B$ (that is the density of the $P^x$ distribution of $X(\tau_B)$ \cite{BGR1961}). By Lemma 3.1 in \cite{BKN2002} we have
\begin{eqnarray}
\nonumber
\left|\nabla \left(E^{x_0} G_r(qu)(X_{\tau_B})\right)\right|
&=& \left|\nabla_x \int_{B^c} P_B(x_0,z) G_r(qu)(z) \, dz \right| \\
\nonumber
&\le& \int_{B^c} |\nabla_x P_B(x_0,z)| G_r(|q|u)(z) \, dz \\
\nonumber
&\le& \frac{c}{\delta_{\Delta_r}(x_0)} \int_{B^c} P_B(x_0,z) G_r(|q|u)(z) \, dz \\
\label{normGrqu}
&=& \frac{c}{\delta_{\Delta_r}(x_0)} E^{x_0} G_r(|q|u)(X_{\tau_B}),
\end{eqnarray}
where $c = c(d,\alpha)$.

By (\ref{Grsplit}) for $|q|$ instead of $q$ we obtain that $E^{x_0} G_r(|q|u)(X_{\tau_B}) \le G_r(|q|u)(x_0)$. Using this and (\ref{eps1}), (\ref{eps2}) we obtain that (\ref{normGrqu}) is bounded from above by $c \eps u(x_0)/\delta_{\Delta_r}(x_0)$, where $c = c(d,\alpha)$.

By Theorem \ref{mainthm} for any $x,y \in B$ we have
$$
|u(x) - u(y)| \le \frac{c}{\delta_{\Delta_r}(x_0)} \left(\sup_{z \in B} u(z) \right) |x - y|,
$$
for some $c = c(d,\alpha,\eta,q)$. Using this and (\ref{qbox}) for any $x,y \in B$ we get
\begin{eqnarray*}
|q(x) u(x) - q(y) u(y)| &\le&
|q(x)| |u(x) - u(y)| + |q(x) -q(y)| |u(y)| \\
&\le& \frac{c (A + 1)}{\delta_{\Delta_r}(x_0)} 
\left(\sup_{z \in B} u(z) \right) |x - y|^{\eta},
\end{eqnarray*}
for some $c = c(d,\alpha,\eta,q)$.

Hence by Lemma \ref{gradientGreen} we obtain
\begin{equation}
\label{GBqu}
|\nabla G_B(qu)(x_0)| \le \frac{c (A + 1)}{\delta_{\Delta_r}(x_0)} 
\left(\sup_{z \in B} u(z) \right) (\delta_{\Delta_r}(x_0))^{\eta + \alpha - 1} (1 + |\log(\delta_{\Delta_r}(x_0))|),
\end{equation}
where $c = c(d,\alpha,\eta,q)$. 

Note that $\delta_{\Delta_r}(x_0) \le r \le \kappa$. By our choice of $\kappa$ we have 
$$
(\delta_{\Delta_r}(x_0))^{\eta + \alpha - 1} (1 + |\log(\delta_{\Delta_r}(x_0))|) \le \eps.
$$

Using this and the Harnack inequality (see (\ref{Harnack}) with $y$ changed to $z$ and $z$ changed to $x_0$) we obtain that the right-hand side of (\ref{GBqu}) is bounded from above by $c \eps u(x_0)/\delta_{\Delta_r}(x_0)$, where $c = c(d,\alpha,\eta,q)$.
\end{proof}

The next lemma is similar to Lemma 5.6 in \cite{BKN2002}. 
\begin{lemma}
\label{lowerbox}
Let $\alpha \in (0,1]$. There are constants $c = c(d,\alpha,\lambda)$, $h = h(d,\alpha,\lambda)$ and $r_1 = r_1(d,\alpha, \lambda, \eta, q, s_0)$ such that if $0 < r \le r_1$ and $u$ is nonnegative in $\R^d$, $q$-harmonic and bounded in $\Delta_r$ and vanishes in $\nabla(Q,r,r)$ then
$$
|\nabla u(x)| \ge c \frac{u(x)}{\delta_{\Delta_r}(x)}, \quad \quad x \in \Delta(Q,r h,r/2).
$$
\end{lemma}
\begin{proof}
The function $u$ satisfies (\ref{representation}) with $W = \Delta_r$. Using \cite[Lemma 4.5]{BKN2002} and scaling, (\ref{eps1}) and Lemma \ref{uppereps} we obtain the result by an appropriate choice of $\eps$ in Lemma \ref{uppereps}.
\end{proof}

\begin{proof}[Proof of Theorem \ref{sharpthm}]
The upper bound follows from Theorem \ref{mainthm}. The lower bound follows from Lemma \ref{lowerbox} and compactness of $\partial D \cap K$.
\end{proof}

\section{Applications}

As an application of the main results of this paper we obtain gradient estimates of eigenfunctions of the fractional Schr{\"o}dinger operator. 

\begin{corollary}
\label{Schrodinger}
Assume that $\alpha \in (0,2)$, $D \subset \R^d$ is an open bounded set,  $q \in \calJ^{\alpha -1}$ when $\alpha \in (1,2)$, or $q$ is H{\"o}lder continuous on $D$ with H{\"o}lder exponent $\eta > 1 - \alpha$ when $\alpha \in (0,1]$. Let $\{\varphi_n\}_{n =1}^{\infty}$ be the eigenfunctions of the eigenvalue problem (\ref{spectral1})-(\ref{spectral2}) for the fractional Schr{\"o}dinger operator on $D$ with zero exterior condition. Then $\nabla \varphi_n(x)$ exist for any $n \in \N$, $x \in D$ and we have
\begin{equation}
\label{Schr1}
|\nabla \varphi_1(x)| \le c \frac{\varphi_1(x)}{\delta_D(x) \wedge 1}, \quad \quad x \in D,
\end{equation}
where $c = c(D,q,\alpha,\eta)$ and
\begin{equation}
\label{Schr2}
|\nabla \varphi_n(x)| \le  \frac{c_n}{\delta_D(x) \wedge 1}, \quad \quad x \in D,
\end{equation}
where $c_n = c_n(D,q,\alpha,\eta)$. Furthermore, if additionally $D \subset \R^d$ is a bounded Lipschitz domain then there exists $\eps = \eps(D,q,\alpha,\eta)$ such that 
\begin{equation}
\label{Schr3}
|\nabla \varphi_1(x)| \ge c \frac{\varphi_1(x)}{\delta_D(x)}, \quad \quad x \in D, \quad \delta_D(x) \le \eps,
\end{equation}
where $c = c(D,q,\alpha,\eta)$. 
\end{corollary}
The result is new even for $q \equiv 0$. In that case this is the eigenvalue problem for the fractional Laplacian with zero exterior condition. This eigenvalue problem have been recently very intensively studied see e.g. \cite{BK2004}, \cite{CS2005}, \cite{BKS2009}, \cite{KL2011}, \cite{FG2011}, \cite{BKM2006}. 

For $\alpha = 2$, under additional assumptions that $d \ge 3$, $D$ is connected and Lipschitz, inequalities (\ref{Schr1}), (\ref{Schr2}) follows from \cite[Theorem 1]{CZ1990} and inequality (\ref{Schr3}) follows from \cite[Theorem 1]{BP1999}. 

Before we come to the proof of Corollary \ref{Schrodinger} we will need the following easy addendum to the results obtained in \cite{BKN2002}.
\begin{lemma}
\label{qharmonic12}
Let $\alpha \in (1,2)$, $q \in \calJ^{\alpha - 1}$ and $D \subset \R^d$ be an open set. Assume that the function $u$  is $q$-harmonic in $D$ and $\|u\|_{\infty} < \infty$. Then $\nabla u(x)$ exists for any $x \in D$ and we have 
$$
|\nabla u(x)| \le c \frac{\|u\|_{\infty}}{\delta_{D}(x) \wedge 1}, \quad \quad x \in D,
$$
where $c = c(d,\alpha,q)$.
\end{lemma}
\begin{proof}
The proof of this lemma follows from the arguments used in \cite{BKN2002}. First note that the assertion of Lemma 5.4 in \cite{BKN2002} remains true if we replace the assumption that $u$ is nonnegative in $\R^d$ by the assumption that $\|u\|_{\infty} < \infty$ and when we replace $u(x)$ by $\|u\|_{\infty}$ on the right-hand side of the estimate of $|\nabla G_r(qu)(x)|$. Then the proof of Lemma \ref{qharmonic12} is almost the same as the proof of Lemma 5.5 in \cite{BKN2002}.
\end{proof}

\begin{proof}[proof of Corollary \ref{Schrodinger}]
It is clear that $\varphi_n$ is not $(q + \lambda_n)$-harmonic on the whole $D$ because $(D,q+\lambda_n)$ is not gaugeable. However by the definition of the Kato class and standard arguments (see e.g. page 299 \cite{BB2000}) for any $n = 1,2,\ldots$ there exists $r \in (0,1]$ and the finite number of balls $B(x_1,r),\ldots,B(x_M,r)$ such that $x_1,\ldots,x_M \in D$,
$$
D \subset \sum_{m = 1}^{M} B(x_m,r)
$$
and each $(B(x_m,2r) \cap D,q+\lambda_n)$ is gaugeable. This means that $\varphi_n$ is $(q + \lambda_n)$-harmonic on each $B(x_m,2r) \cap D$. Note that for any $x \in B(x_m,r) \cap D$ we have
$$
\delta_{B(x_m,2r) \cap D}(x) \wedge 1 \ge \delta_D(x) \wedge r \wedge 1 \ge r (\delta_D(x) \wedge 1).
$$
Now, (\ref{Schr1}), (\ref{Schr2}) follow from Theorem \ref{mainthm} for $\alpha \in (0,1]$ and from \cite[Lemma 5.5]{BKN2002}, Lemma \ref{qharmonic12} for $\alpha \in (1,2)$. Inequality (\ref{Schr3}) follows from similar arguments and Lemma \ref{lowerbox} for $\alpha \in (0,1]$ and \cite[Lemma 5.6]{BKN2002} for $\alpha \in (1,2)$.
\end{proof}

As another application of our main result we show that under some assumptions on $q$ a weak solution of $\Delta^{\alpha/2} u + q u = 0$ is in fact a strong solution. First we need the following easy lemma.

\begin{lemma} 
\label{existence}
Let $\alpha \in (0,1)$. Choose $x_0 \in \R^d$ and $r > 0$. Assume that a Borel function $u: \R^d \to \R$ satisfies 
$$
\int_{\R^d} \frac{|u(y)|}{(1 + |y|)^{d + \alpha}} \, dy < \infty,
$$
$\nabla u(x)$ exists and $|\nabla u(x)| \le A$ for all $x \in B(x_0,r)$ and some constant $A$. Then $\Delta^{\alpha/2} u(x)$ is well defined and continuous on $B(x_0,r/2)$.
\end{lemma}
\begin{proof}
Let $x \in B(x_0,r/2)$. Choose $\eps \in (0,r/2)$. We have
\begin{eqnarray}
\label{eps1a}
\int_{|x - y| < \eps} \frac{|u(y) - u(x)|}{|y - x|^{d + \alpha}} \, dy
&\le& \int_{|x - y| < \eps} \frac{A |y - x|}{|y - x|^{d + \alpha}} \, dy \\
\label{eps2a}
&=& c A \int_0^{\eps} \rho^{-\alpha} \, d \rho \to 0, \quad \quad \text{when} \quad \eps \to 0,
\end{eqnarray}
where $c = c(d)$. By the definition of $\Delta^{\alpha/2}$ (see Preliminaries) we obtain that $\Delta^{\alpha/2} u(x)$ is well defined.

One can easily show that for any fixed $\eps \in (0,r/2)$ the function
$$
f_{\eps}(x) = \int_{|x - y| > \eps} \frac{|u(y) - u(x)|}{|y - x|^{d + \alpha}} \, dy
$$
is continuous on $B(x_0,r/2)$. This and (\ref{eps1a}) - (\ref{eps2a}) imply that $\Delta^{\alpha/2} u(x)$ is continuous on $B(x_0,r/2)$.
\end{proof}

\begin{proof}[proof of Corollary \ref{weakstrong}]
Choose arbitrary $x_0 \in D$. It is clear that there exists $r > 0$ such that $B(x_0,2r) \subset \subset D$ and $(B(x_0,2r),q)$ is gaugeable. This can be done by Khasminski's lemma (see page 299 \cite{BB2000}). Put $B = B(x_0,r)$. By \cite[Theorem 5.5]{BB1999} we may assume that $u$ is a $q$-harmonic function on $B(x_0,2r)$ (after a modification on a set of Lebesgue measure zero). By (\ref{representation}) we get
\begin{equation}
\label{repr1}
u(x) = E^x u(X_{\tau_B}) + G_B(qu)(x), \quad \quad  x \in \R^d.
\end{equation}
By Theorem \ref{mainthm} and Lemma \ref{existence} $\Delta^{\alpha/2} u(x)$ is well defined and continuous on $B(x_0,r/2)$. The function $v(x) = E^x u(X_{\tau_B})$ is an $\alpha$-harmonic function on $B(x_0,r/2)$, so $\Delta^{\alpha/2} v(x) = 0$ on $B(x_0,r/2)$. Hence by (\ref{repr1}) we obtain
$$
\Delta^{\alpha/2} u(x) = \Delta^{\alpha/2}(G_B(qu))(x), \quad \quad x \in B(x_0,r/2).
$$
By Lemma 5.3 \cite{BB2000} we have
$$
\Delta^{\alpha/2}(G_B(qu))(x) = -q(x) u(x),
$$
for almost all $x \in B(x_0,r/2)$. But both sides of this equality are continuous so in fact this equality holds for all $x \in B(x_0,r/2)$.
\end{proof}

{\bf{ Ackowledgements.}} I am grateful for the hospitality of the Institute of Mathematics, Polish Academy of Sciences, the branch in Wroc{\l}aw, where part of this paper was written.

\end{document}